\documentclass[a4, draft]{amsart} 
\usepackage{enumerate}
\usepackage{color}
\usepackage{amssymb}
\usepackage{braket}

\numberwithin{equation}{section}

\theoremstyle{plain}
\newtheorem{thm}{Theorem}
\newtheorem{prop}{Proposition}
\newtheorem{lem}{Lemma}

\theoremstyle{definition}

\newtheorem{rem}{Remark}

\newcommand\md{\textrm{\upshape mod }}

\DeclareMathOperator*{\dsum}{\sum\sum}

\newcommand{\plm}{\pm}
\newcommand{\veps}{\varepsilon}
\newcommand{\ph}{\varphi}

\allowdisplaybreaks
\usepackage{blindtext}
\usepackage{comment}

\begin{document}
\title[Goldbach Representations in Arithmetic Progressions]
{Goldbach Representations in Arithmetic Progressions and zeros of Dirichlet $L$-functions}
\author{Gautami Bhowmik}
\address{G. Bhowmik: Laboratoire Paul Painlev{\'e}, Labex-Cempi, Universit{\'e} Lille 1, 59655
Villeneuve d'Ascq Cedex, France}
\email{bhowmik@math.univ-lille1.fr}

\author{Karin Halupczok}
\address{K. Halupczok: Mathematisch-Naturwissenschaftliche Fakult{\"a}t
Heinrich-Heine-Universit{\"a}t D{\"u}sseldorf, Universit{\"a}tsstr. 1,
40225 D{\"u}sseldorf, Germany}
\email{karin.halupczok@uni-duesseldorf.de}

\author{Kohji Matsumoto}
\address{K. Matsumoto: Graduate School of Mathematics, Nagoya University, Furocho,
Chikusa-ku, Nagoya 464-8602, Japan}
\email{kohjimat@math.nagoya-u.ac.jp}

\author{Yuta Suzuki}
\address{Y. Suzuki: Graduate School of Mathematics, Nagoya University, Furocho,
Chikusa-ku, Nagoya 464-8602, Japan}
\email{m14021y@math.nagoya-u.ac.jp}

\thanks{The first and third authors benefit from the financial support of the French-Japanese Joint Project
"Zeta-functions of Several Variables and Applications" (PRC CNRS/JSPS 2015-2016).
The fourth author is supported by Grant-in-Aid for JSPS Research Fellow (Grant Number: JP16J00906)
and had the partial aid of CEMPI for his stay at Lille.}
\date{}
\maketitle

\begin{abstract}
Assuming a conjecture on distinct zeros of Dirichlet $L$-functions
we get asymptotic results on the average number of representations of an integer as the sum of two primes in arithmetic progression.
On the other hand the existence of good error terms gives information
on the location of zeros of $L$-functions.
Similar results are obtained for an integer in a congruence class expressed as the sum of two primes.
\end{abstract}

\section{Introduction and Results}\label{Intro}
The  Goldbach problem of representing every even integer larger than 2 as the sum of two primes has several variants,
one being that in which the summands are primes in given arithmetic progressions.
Similar to the original problem it is known that almost all even integers satisfying some congruence condition
can be written as the sum of two primes in congruence classes.
Quantitatively, the exceptional set of integers less than $X$ and satisfying the necessary congruence condition,
which can not be written as the sum of primes congruent to a common modulus $q$
may be estimated as $O(\varphi (q)^{-1}X^{1-\delta})$ for a computable positive constant $\delta$ and all $q\le X^{\delta}$ \cite{LZ97}.
(See \cite{Bauer16} for more recent results.)

Though the complete solution of these binary Goldbach problems is out of sight,
the related question of the average number of representations of integers as sums of  primes seems more accessible.
The study of the average order of the weighted function
\[
G(n)=\sum_{\ell+m=n}\Lambda(\ell)\Lambda(m)
\]
where $\Lambda$ is the von Mangoldt function has begun with Fujii \cite{Fuj91} and continues to be actively pursued.
However the current state of knowledge on the zeros of the Riemann zeta function $\zeta(s)$ is not enough
to obtain ``good'' error terms unconditionally and the Riemann Hypothesis is always assumed in such studies.
In fact obtaining  sufficiently sharp error terms for average orders of the mean value of $G(n)$
is expected to solve other conjectures like the Riemann Hypothesis,
as elaborated by Granville \cite{Gra07} in the classical case of unrestricted primes.
This paper is an analogous study with the two primes in arithmetic progressions with a common modulus.

The function that we consider here, with $a,b$  positive integers coprime to $q$, is 
\[
G(n;q,a,b)=\sum_{\substack{\ell+m=n\\\ell\equiv a,\,m\equiv b\,(\md q)}}\Lambda(\ell)\Lambda(m)
\]
whose summatory function  defined as
\[
S(x;q,a,b)=\sum_{n\leq x}G(n;q,a,b)
\]
was introduced by R{\"u}ppel \cite{Rup12} and further studied by the fourth author \cite{SuzIJNT}. 

On the lines of Granville we consider the relations between an explicit formula for $S(x;q,a,b)$ and zeros of $L$-functions.
In \cite[1A]{Gra07} it is stated that there is an equivalence between the estimate
\begin{equation}\label{Gra-1A}
\sum_{n\leq x}(G(n)-J(n))\ll x^{3/2+o(1)}
\end{equation}
and the Riemann Hypothesis (RH) for $\zeta(s)$
where $J(n)=0$ for odd $n$ and,
with $C_2=2\prod_{p>2} (1-\frac{1}{(p-1)^{2}})$ being the twin prime constant,
\[
J(n)
=
n\cdot C_2\prod_{\substack{p\mid n\\p>2}}\frac{p-1}{p-2}
\]
for even $n$. The function $J(n)$ is believed since Hardy and Littlewood to be a good approximation for $G(n)$ (cf.\ \cite{HL24}).

We denote
by $\chi$ a Dirichlet character $(\md q)$,
by $L(s,\chi)$ the associated Dirichlet $L$-function and
by $\rho_{\chi}$ its non-trivial zeros.
Let $B_{\chi}=\sup\{\Re \rho_{\chi}\}$ and $B_q=\sup\{B_{\chi}\mid\chi\,(\md q)\}$.
Hence $1/2\leq B_q\leq 1$ for $q\ge 1$.
In case of the trivial character, we use $\rho$ for non-trivial zeros of $\zeta(s)$, 
and $B=\sup\{\Re\rho\}$.

In the context of primes in congruence classes we first need to formulate the Distinct Zero Conjecture (DZC) on zeros of $L$-functions as:
\begin{quote}
\textit{For any $q\geq 1$, any two distinct Dirichlet $L$-functions associated with characters of modulus $q$
do not have a common non-trivial zero, except for a possible multiple zero at $s=1/2$.}
\end{quote}
Though weaker than the non-coincidence conjecture found in literature that expects all zeros of all primitive $L$-functions to be linearly independent except for the possible multiple zero at $s=1/2$
(cf. \cite[p.353]{Con03}), this suffices for our purpose.

\begin{thm}\label{MainTheorem-1}
Let $a,b$ be integers with $(ab,q)=1$.
\begin{enumerate}
\renewcommand{\labelenumi}{{\upshape (\arabic{enumi})}}
\item
For $x\ge2$, we have
\begin{equation}
\label{MainTheorem-1-asymp}
S(x;q,a,b)=\frac{x^2}{2\varphi(q)^2}+O(x^{1+B_q}),
\end{equation}
where the implicit constant is absolute.
\item
Let {\upshape DZC} be true, let $\chi(a)+\chi(b)\neq 0$ for all characters
$\chi\ (\md q)$ and let
$1/2\le d<1$.
If the asymptotic formula
\begin{align}\label{MainTheorem-1-formula}
S(x;q,a,b)=\frac{x^2}{2\varphi(q)^2}+O_q(x^{1+d+\varepsilon}).
\end{align}
holds for any $\varepsilon>0$, then $B_q\le d$ or $B_q=1$.
Further if \eqref{MainTheorem-1-formula} holds with $a=b$, then $B_{q}\le d$. 
\end{enumerate}
\end{thm}

\begin{rem}
Our result thus falls short of an equivalence
with the Generalized Riemann Hypothesis (GRH) for $L$-functions modulo $q$
since we have an additional possibility of $B_q=1$.
Using a yet unpublished idea of I.~Ruzsa, we were able to exclude this
possibility for the case $a=b$ under the DZC .

Thus  the proof of the equivalence
between the RH and (\ref{Gra-1A}) is now complete (see also \cite{BRu}).
All other equivalences  for primes in arithmetic progressions are still
partial.

In particular if $B_{q}=1$, 
then the error term of Theorem \ref{Asymptotic_B_SP}
becomes so large that it hides the information on non-trivial zeros.
Thus the analytic continuation of the generating function could only be obtained  up to $\sigma>2B_q$
( \ref{Ruppel_continuation}).

\end{rem}

\begin{rem}
We need the condition that $\chi(a)+\chi(b)\neq0$
for all $\chi\ (\md q)$ in order to assure that the residue $r_1(\rho_q)$
in Proposition \ref{Ruppel_continuation} does not vanish.
\end{rem}

To prove Theorem \ref{MainTheorem-1} we need an explicit formula for $S(x;q,a,b)$,
which can be stated as follows.
\begin{thm}\label{Asymptotic_B_SP}
Let $a,b$ be integers with $(ab,q)=1$.
Then,  for $x\ge2$ and for any $\varepsilon>0$ ,
\begin{align}\label{Asymptotic_B_SP-formula}
\lefteqn{S(x;q,a,b)
=
\frac{x^2}{2\varphi(q)^2}}\\
&-\frac{1}{\varphi(q)^2}\sum_{\chi\,(\md q)}(\overline{\chi(a)}+\overline{\chi(b)})
\sum_{\rho_\chi}\frac{x^{\rho_\chi+1}}{\rho_\chi(\rho_\chi+1)}
+O(x^{2B_q^\ast}(\log qx)^5),\nonumber
\end{align}
where the implicit constant is absolute and
\[
B_q^\ast=B_q^\ast(x)=\min(B_q,1-\eta),\quad
\eta=\eta_{q}(x)=\frac{c_1(\varepsilon)}{\min(q^{\varepsilon},(\log x)^{4/5})}
\]
with some small constant $c_1(\varepsilon)>0$ depending only on $\varepsilon>0$.
\end{thm}

First in Section \ref{Asymp1} we  prove an
explicit formula with a weaker error estimate (Proposition \ref{prop1})
using a generalized Landau--Gonek formula for $L$-functions 
(Proposition \ref{LandauGonek} in Section \ref{Landau}).
This weaker form is an analogue of Granville \cite[Corrigendum, (2)]{Gra07},
which states
\begin{equation}
\label{nocongruence}
\sum_{n\le x}G(n)
=\frac{x^2}{2}-2\sum_{\substack{\rho\\|\Im\rho|\leq x}}\frac{x^{\rho+1}}{\rho(\rho+1)}
+O\left(x^{(2+4B)/3}(\log x)^2\right),
\end{equation}
and the proof of Proposition \ref{prop1} essentially runs along the line suggested
by \cite{Gra07}.
Therefore Sections \ref{Asymp1} and \ref{Landau} include a reconstruction of
Granville's argument for the asymptotic order.
However we can go further; we take this opportunity to 
prove (in Sections \ref{Lemma_B_SP} and \ref{Asymp2}) the stronger error estimate
\eqref{Asymptotic_B_SP-formula}, an 
 analogue of that announced in
\cite[(5.1)]{Gra07} (cf. \cite[Corrigendum, comments before (2)]{Gra07}),
using a kind of circle method
of the first author and Schlage-Puchta \cite{B_SP}.

With the help of  Theorem \ref{Asymptotic_B_SP} above the analytic continuation 
of the Dirichlet series 
$$\sum_{n=1}^{\infty}\frac{G(n;q,a,b)}{n^s}$$ 
is examined in Proposition \ref{Ruppel_continuation} (in Section \ref{Dirichlet_series})
and this enables us to establish  relations between the error terms in the average of Goldbach problems in arithmetic progressions and zeros of Dirichlet $L$-functions.

Further, we examine the case of $n$ with modulus conditions as in \cite [1B]{Gra07} where it is stated that the GRH for Dirichlet $L$-functions $L(s,\chi)$, over
all characters $\chi$, the modulus of which are odd squarefree divisors of
$q$, is equivalent to the estimate
\begin{align}\label{Gra-1B}
\sum_{\substack{n\leq x\\n\equiv 2\,(\md q)}}(G(n)-J(n))\ll x^{3/2+o(1)}.
\end{align}
Moreover  in \cite[Corrigendum, 1C]{Gra07} it is stated that if the estimate
\begin{align}\label{Gra-1C}
\sum_{\substack{n\leq x\\q\mid n}}G(n) =
\frac{1}{\varphi(q)}\sum_{n\leq x}G(n)+ O_q(x^{1+o(1)})
\end{align}
is attained then the GRH for Dirichlet $L$-functions $L(s,\chi)$,
$\chi\,(\md q)$ holds; and under this hypothesis  the last estimate would have the error term 
$O(x^{4/3} (\log x) ^2)$.

Here we extend \eqref{Gra-1B} with the general congruence condition
$n\equiv c$ for an arbitrary positive integer $c$ instead of the special case $n\equiv  2$.
Assuming the GRH for  $L$-functions $(\md q)$ 
we can deduce the estimate
\begin{equation}\label{eq:thm1B_}
\sum_{\substack{n\leq x\\n\equiv c\,(\md q)}}(G(n)-J(n))\ll
x^{3/2}.
\end{equation}
However in the other direction, 
we could not deduce satisfactory conclusions on
the size of $B_{q}$
when $a\ne b$.  
In particular we were unable then
to reconstruct the reverse implications for (\ref{Gra-1B})
and (\ref{Gra-1C}). In the following we 
give a further example of a condition
with which we can get the 
reverse implication.

\begin{thm}\label{thmsiegel}
Let $q,c$ be integers such that $(2,q)\mid c$.
\begin{enumerate}
\renewcommand{\labelenumi}{{\upshape (\arabic{enumi})}}
\item For $x\ge2$, we have
\begin{equation}
\label{thmsiegel_asymp}
\sum_{\substack{n\leq x\\n\equiv c\,(\md q)}}(G(n)-J(n))
\ll
x^{1+B_q},
\end{equation}
where the implicit constant is absolute.
\item
Assume that
\begin{equation}
\label{eq:thmsiegel}
\sum_{\substack{n\leq x\\n\equiv c\,(\md q)}}(G(n)-J(n))\ll_q
x^{1+d+\varepsilon}
\end{equation}
holds for some $1/2\leq d\leq 1$ and any $\varepsilon>0$.
If there exists a zero $\rho_0$ of $\prod_{\chi\,(\md q)}L(s,\chi)$ such that
\begin{enumerate}
\renewcommand{\labelenumii}{{\upshape(\alph{enumii})}}
\item $B_q=\Re\rho_0$
\item $\rho_0$ belongs to a unique character $\chi_1\ (\md q)$
\item the conductor $q^\ast$ of $\chi_1\ (\md q)$ is squarefree and satisfies $(c,q^\ast)=1$,
\end{enumerate}
then $B_q=\Re \rho_0\le d$.
\end{enumerate}
\end{thm}

The  conditions on $\rho_0$ might resemble those of the Landau--Siegel zero
although  (2) above is not actually applicable to a Landau--Siegel zero
if there are some complex zeros of $L$-function of the same modulus
which are very close to the vertical line $\sigma=1$.

\newpage
To obtain the above results we require an asymptotic formula with $B_q^\ast$  as in Theorem \ref{Asymptotic_B_SP}.

\begin{thm}
\label{Asymptotic_1B}
For $x\ge2$, $\varepsilon>0$ and for any positive integer $c$ we have
\begin{multline*}
\sum_{\substack{n\le x\\n\equiv c\,(\md q)}}G(n)
=
\frac{\mathfrak{S}_q(c)}{2}x^2
-
\frac{2}{\varphi(q)^2}
\sum_{\substack{a=1\\(a(c-a),q)=1}}^q
\sum_{\chi\,(\md q)}\overline{\chi(a)}
\sum_{\rho_\chi}\frac{x^{\rho_\chi+1}}{\rho_\chi(\rho_\chi+1)}\\
+O(x^{2B_q^\ast}(\log qx)^5),
\end{multline*}
where 

\[
\mathfrak{S}_q(c)
=
\frac{1}{\varphi(q)}\prod_{\substack{p\mid q\\p\nmid c}}\frac{p-2}{p-1},
\]
and the implicit constant is absolute.
\end{thm}
Note that 
$\mathfrak{S}_q(c)=0$
if $(2,q)\nmid c$.
The above theorem is proven in Section \ref{Asymp2}.
To ensure the uniformity of $q$  it is not enough to sum Theorem \ref{Asymptotic_B_SP} over residues and we need other tools like Lemma \ref{pre_lem}
.    Finally, using Theorem \ref{Asymptotic_1B}, we  give the proof of
\eqref{eq:thm1B_} and Theorem \ref{thmsiegel} in Section \ref{Dirichlet_series}.
In Section~\ref{sec:exclusion}, we give the proof for the supplement
of Theorem \ref{MainTheorem-1} (2) in the case $a=b$.

\medskip
\textbf{Acknowledgements.}
The first two authors are grateful to Professor Andrew Granville
for helpful discussions.
Thanks are due to Professor Masatoshi Suzuki for useful information, and to
Professor Keiju Sono and the referee for their valuable comments on
the first version of the manuscript. 
We particularly thank Professor Imre Ruzsa for the  idea that
improves Theorem~\ref{MainTheorem-1} (2) in the case $a=b$. 

\section{Some preliminaries}\label{Preliminary_zero}
In this section we fix notation on the zeros of Dirichlet $L$-functions
and give some basic lemmas on Dirichlet $L$-functions.
Results used directly from \cite{MV} are only cited but in other cases 
details are added . 

As we mentioned in Section \ref{Intro},
we denote by $\chi\ (\md q)$ a Dirichlet character $(\md q)$
and by $\chi^\ast\ (\md q^\ast)$ the primitive character inducing $\chi\ (\md q)$.
If there is no specific mention, any statement with $\chi\ (\md q)$ is stated for any $q\ge1$ and any character $\chi\ (\md q)$.

We denote the Dirichlet $L$-function associated to $\chi\ (\md q)$ by $L(s,\chi)$.
We say a zero of $L(s,\chi)$ is non-trivial, if it is contained in the strip $0<\sigma<1$.
We denote by $\rho_\chi$ non-trivial zeros of $L(s,\chi)$ with the real part $\beta_\chi$ and the imaginary part $\gamma_\chi$.
As a summation variable, the letter $\rho_\chi$ runs through all non-trivial zeros of $L(s,\chi)$ counted with multiplicity. We denote the Landau--Siegel zero of $(\md q)$ by $\beta_1$.

We let $\delta_0(\chi)=1$ if $\chi=\chi_0$ is the principal character and $\delta_0(\chi)=0$ otherwise.
Similarly we let $\delta_1(\chi)=1$ if $\chi$ is the exceptional character
(that is, whose $L$-function has a Landau--Siegel zero) and $\delta_1(\chi)=0$ otherwise.

We first evoke some lemmas for sums over non-trivial zeros.

\begin{lem}[{\cite[Theorem 10.17]{MV}}]
\label{Lem:zero_counting_local}
For any $T\ge0$, we have
\[
\sum_{\substack{\rho_\chi\\T\le|\gamma_\chi|\le T+1}}1\ll\log q(|T|+2).
\]
\end{lem}

\begin{lem}
\label{Lem:zero_sum}
For any $T\ge1$ and $\chi\ (\md q)$, we have
\[
\sum_{\substack{\rho_\chi\neq1-\beta_1\\|\gamma_\chi|\le T}}\frac{1}{|\rho_\chi|}\ll(\log2qT)^2,\quad
\sum_{\substack{\rho_\chi\\|\gamma_\chi|\le T}}\frac{1}{|\rho_\chi|}
\ll(\log 2qT)^2+\delta_1(\chi)q^{1/2}(\log q)^2.
\]
\end{lem}
\begin{proof}
For the first estimate we dissect the sum at $|\gamma_\chi|=1$ as
\[
\sum_{\substack{\rho_\chi\neq1-\beta_1\\|\gamma_\chi|\le T}}\frac{1}{|\rho_\chi|}
=
\sum_{\substack{\rho_\chi\neq1-\beta_1\\|\gamma_\chi|\le 1}}\frac{1}{|\rho_\chi|}
+
\sum_{\substack{\rho_\chi\\1<|\gamma_\chi|\le T}}\frac{1}{|\rho_\chi|}.
\]
For the first sum  Lemma \ref{Lem:zero_counting_local} gives
\begin{align*}
\sum_{\substack{\rho_\chi\neq1-\beta_1\\|\gamma_\chi|\le 1}}\frac{1}{|\rho_\chi|}
&\le
\sum_{\substack{\rho_\chi\neq1-\beta_1\\|\gamma_\chi|\le 1}}\frac{1}{\beta_\chi}
=
\sum_{\substack{\rho_\chi\neq\beta_1\\|\gamma_\chi|\le 1}}\frac{1}{1-\beta_\chi}\\
&\ll
\sum_{\substack{\rho_\chi\\|\gamma_\chi|\le 1\\\beta_\chi>1-c_0/\log2q}}\frac{1}{1-\beta_\chi}
\ll
(\log2q)\sum_{\substack{\rho_\chi\\|\gamma_\chi|\le1}}1\ll(\log 2q)^2\ll(\log 2qT)^2,
\end{align*}
where $c_0>0$ is some small absolute constant while
for the second sum, Lemma \ref{Lem:zero_counting_local} gives
\[
\sum_{\substack{\rho_{\chi}\\1<|\gamma_{\chi}|\leq T}}\frac{1}{|\rho_{\chi}|}
\leq \sum_{n\leq T}\sum_{\substack{\rho_{\chi}\\ n<|\gamma_{\chi}|\leq n+1}}
\frac{1}{|\rho_{\chi}|}
\ll \sum_{n\leq T}\frac{\log q(n+2)}{n}
\ll (\log2qT)^2.
\]
Thus the first estimate follows.
The second estimate is obtained by the first estimate
combined with the well-known bound \cite[Corollary 11.12]{MV} of the Siegel zero
\begin{equation}
\label{EQ:Landau_Page_bound}
\beta_1>1-\frac{c_2}{q^{1/2}(\log q)^2},
\end{equation}
where $c_2>0$ is some absolute constant.
\end{proof}

\begin{lem}
\label{Lem:zero_sum2}
For any $T\ge1$ and $\chi\ (\md q)$, we have
\[
\sum_{\substack{\rho_\chi\\|\rho_\chi|>T}}\frac{1}{|\rho_\chi|^2}\ll\frac{\log2qT}{T},\quad
\sum_{\rho_\chi}\frac{1}{|\rho_\chi(\rho_\chi+1)|}\ll(\log 2q)^2+\delta_1(\chi)q^{1/2}(\log q)^2.
\]
\end{lem}
\begin{proof}
The first estimate is again obtained by using Lemma \ref{Lem:zero_counting_local} ,i.e.
\begin{align*}
\sum_{\substack{\rho_\chi\\|\gamma_\chi|>T}}\frac{1}{|\rho_\chi|^2}
\le
\sum_{n=[T]}^\infty\sum_{\substack{\rho_\chi\\n<|\gamma_\chi|\le n+1}}\frac{1}{|\rho_\chi|^2}
\ll
\sum_{n=[T]}^\infty\frac{\log qn}{n^2}
\ll
\frac{\log qT}{T},
\end{align*}
whereas the last estimate can be obtained by  comparison to an integral.
For the latter estimate, we combine the former one with Lemma \ref{Lem:zero_sum}.
This gives
\[
\sum_{\rho_\chi}\frac{1}{|\rho_\chi(\rho_\chi+1)|}
\le
\sum_{\substack{\rho_\chi\\|\gamma_\chi|\le 1}}\frac{1}{|\rho_\chi|}
+
\sum_{\substack{\rho_\chi\\|\gamma_\chi|>1}}\frac{1}{|\rho_\chi|^2}
\ll
(\log 2q)^2+\delta_1(\chi)q^{1/2}(\log q)^2.
\]

\end{proof}

We next prepare some explicit formulas for the sum
\[
\psi(x,\chi)=\sum_{n\le x}\chi(n)\Lambda(n).
\]
The next result appears in \cite{MV}
with the  restriction that $\chi\ (\md q)$  be primitive.
\begin{lem}
\label{Lem:explicit}
For any $u,T\ge2$, the explicit formula
\[
\psi(u,\chi)=\delta_0(\chi)u
-\sum_{\substack{\rho_{\chi} \\ |\gamma_{\chi}|\leq T}}
\frac{u^{\rho_{\chi}}}
{\rho_{\chi}}+C(\chi^\ast)+
E(u,T,\chi)
\]
holds, where
\[
E(u,T,\chi)\ll(\log 2q)(\log u)+\frac{u}{T}(\log quT)^2
\]
and $C(\chi)$ is some constant depending only on $\chi$.
\end{lem}
\begin{proof}
For primitive $\chi$, this follows immediately from Theorem 12.5 and 12.10 of \cite{MV} as below.
If $\chi$ is trivial so that $q=1$, we use Theorem 12.5 of \cite{MV} with $x=u$.
We can estimate the last three terms on the right-hand side of (12.3) of \cite{MV} as
\[
-\log 2\pi-\frac{1}{2}\log(1-1/u^2)+R(u,T)\ll(\log 2q)(\log u)+\frac{u}{T}(\log quT)^2,
\]
by using (12.4) of \cite{MV}. This gives the assertion for the case when $\chi$ is trivial.
If $\chi$ is non-principal, we use Theorem 12.10 of \cite{MV} with $x=u$.
The last four terms on the right-hand side of (12.6) of \cite{MV} can be rewritten as
\[
-\frac{1}{2}\log(u-1)-\frac{\chi(-1)}{2}\log(u+1)+C(\chi)+R(u,T;\chi)
\]
\[
=
C(\chi)+O\left((\log 2q)(\log u)+\frac{u}{T}(\log quT)^2\right).
\]
This gives the assertion for the case of $\chi$ being primitive and non-principal.

If $\chi$ is imprimitive, then it suffices to note that the non-trivial zeros of $L(s,\chi)$
are those of $L(s,\chi^\ast)$ and that
\begin{equation}
\label{EQ:remove_coprimality}
\begin{aligned}
\psi(u,\chi)-\psi(u,\chi^\ast)
\ll
\sum_{\substack{n\le u\\(n,q)>1}}\Lambda(n)
&=
\sum_{p\mid q}(\log p)\left[\frac{\log u}{\log p}\right]\\
&\le
\left(\frac{\log u}{\log 2}\right)\sum_{p\mid q}(\log p)\le(\log 2q)(\log u),
\end{aligned}
\end{equation}
which is absorbed into $E(u,T,\chi)$. 
\end{proof}

When we substitute the above explicit formula into some sum or integral,
we need to use a uniform parameter $T$ and a uniform bound of the error term.
Also it is convenient to work  with the case $0\le u<2$.
Thus we modify the above explicit formula in the following form.
\begin{lem}\label{lem-1}
For any $x\geq T\geq 2$ and $x\ge u\ge0$, the explicit formula
\begin{align}\label{1-3}
\psi(u,\chi)=\delta_0(\chi)u-\sum_{\substack{\rho_{\chi} \\ |\gamma_{\chi}|\leq T}}
\frac{u^{\rho_{\chi}}}{\rho_{\chi}}+O\left(\frac{x}{T}(\log qx)^2+\delta_1(\chi)q^{1/2}(\log q)^2\right)
\end{align}
holds.
\end{lem}

\begin{proof}
We first consider the case $u\ge2$.
For this  we use Lemma \ref{Lem:explicit}.
Since $u\le x\le T$, the error term of Lemma \ref{Lem:explicit} is
\[
E(u,T,\chi)
\ll(\log 2q)(\log x)+\frac{x}{T}(\log qxT)^2
\ll\frac{x}{T}(\log qxT)^2.
\]
Further, by Theorem 11.4 of \cite{MV}, we know that 
$(L'/L)(1,\overline{\chi^*})\ll \log 2q$ 
if $\chi$ is not exceptional, while
for the exceptional $\chi$ and exceptional zero $\beta_{1}$,
\eqref{EQ:Landau_Page_bound} gives
\[
\frac{L'}{L}(1,\overline{\chi^*})=\frac{1}{1-\beta_{1}} + O(\log 2q) \ll q^{1/2}(\log q)^2.
\]
Therefore by (12.7) of \cite{MV}, we obtain
\begin{equation}
\label{EQ:D_estimate}
C(\chi^\ast)
\ll \log 2q+\delta_1(\chi)q^{1/2}(\log q)^2
\ll \frac{x}{T}(\log qx)^2+\delta_1(\chi)q^{1/2}(\log q)^2
\end{equation}
for non-principal $\chi$, from which the lemma follows for the case $u\ge2$.

The remaining case is when $0\leq u<2$.
Now the sum on the right-hand side of \eqref{1-3} is estimated by using Lemma \ref{Lem:zero_sum} as
\begin{align*}
\sum_{\substack{\rho_{\chi} \\ |\gamma_{\chi}|\leq T}}
\frac{u^{\rho_{\chi}}}{\rho_{\chi}}
\ll u\sum_{\substack{\rho_\chi\\|\gamma_\chi|\le T}}\frac{1}{|\rho_\chi|}
&\ll (\log qT)^2+\delta_1(\chi)q^{1/2}(\log q)^2\\
&\ll \frac{x}{T}(\log qx)^2+\delta_1(\chi)q^{1/2}(\log q)^2.
\end{align*}
since $T\le x$. On the other hand, when $u<2$, the left-hand side of \eqref{1-3} is zero.
Therefore the assertion holds trivially.
\end{proof}

\section{An asymptotic formula for $S(x;q,a,b)$}\label{Asymp1}

In this section, we deduce a prototype of Theorem \ref{Asymptotic_B_SP} along the line of \cite{Gra07}:

\begin{prop}\label{prop1}
For integers $a,b,q$ with $(ab,q)=1$, we have
\begin{multline*}
S(x;q,a,b)=\frac{x^2}{2\varphi(q)^2}
-\frac{1}{\varphi(q)^2}\sum_{\chi\,(\md q)}(\overline{\chi(a)}+\overline{\chi(b)})
\sum_{\rho_{\chi}}
\frac{x^{\rho_{\chi}+1}}{\rho_{\chi}(\rho_{\chi}+1)}\\
+O\left(x^{(2+4B_q)/3}(\log qx)^4\right),
\end{multline*}
where the implicit constant is absolute.
\end{prop}

\noindent The proof is divided into three parts.

{\it Step 1} : The first substitution in the explicit formula.

With
\[
\psi(x;q,a)=\sum_{\substack{m\leq x\\m\equiv a\,(\md q)}}\Lambda(m),
\]
we can write
\begin{align}\label{1-1bis}
S(x;q,a,b)=
\sum_{\substack{\ell\leq x\\ \ell\equiv a\,(\md q)}}\Lambda(\ell)\psi(x-\ell;q,b).
\end{align}
Using the orthogonality relation of Dirichlet characters,
we have
\[\label{1-2}
\psi(x;q,b)
=\frac{1}{\varphi(q)}\sum_{\chi\,(\md q)}\overline{\chi(b)}\psi(x,\chi).
\]
We substitute the explicit formula given by Lemma \ref{lem-1} here.
This gives
\begin{equation}\label{1-8}
\psi(u;q,b)
=\frac{1}{\varphi(q)}A(u,T;q,b)+B(u,T;q,b)
\end{equation}
for $x\ge T\ge2$ and $x\ge u\ge0$, where
\[
A(u,T;q,b)=u-\sum_{\chi\,(\md q)}\overline{\chi(b)}
\sum_{\substack{\rho_{\chi} \\ |\gamma_{\chi}|\leq T}}
\frac{u^{\rho_{\chi}}}{\rho_{\chi}}
\]
and $B(u,T;q,b)$ is the error term satisfying
\[
B(u,T;q,b)
\ll
\frac{x}{T}(\log qx)^2+\frac{1}{\varphi(q)}\sum_{\chi\,(\md q)}\delta_1(\chi)q^{1/2}(\log q)^2
\ll\frac{x}{T}(\log qx)^2,
\]
where for estimating the term involving $\delta_1(\chi)q^{1/2}(\log q)^2$,
we use the fact that there is at most one exceptional character $(\md q)$.
Substituting \eqref{1-8} into \eqref{1-1bis}, we obtain
\begin{equation}\label{1-9}
S(x;q,a,b)=\frac{1}{\varphi(q)}\sum_{\substack{\ell\leq x\\\ell\equiv a\,(\md q)}}
\Lambda(\ell)A(x-\ell,T;q,b)
+O\left(\frac{x^2}{T}(\log qx)^2\right).
\end{equation}

{\it Step 2} : The second substitution in the explicit formula.

Now we evaluate the sum on the right-hand side of \eqref{1-9},
we split it into two parts.
\begin{align}\label{2-1}
&\ \frac{1}{\varphi(q)}\sum_{\substack{\ell\leq x \\ \ell\equiv a\,(\md  q)}}
\Lambda(\ell)A(x-\ell,T;q,b)\\
=&\ 
\frac{1}{\varphi(q)}\sum_{\substack{\ell\leq x \\ \ell\equiv a\,(\md  q)}}
\Lambda(\ell)(x-\ell)\notag\\
&\qquad-\frac{1}{\varphi(q)}\sum_{\substack{\ell\leq x \\ \ell\equiv a\,(\md  q)}}
\Lambda(\ell)\sum_{\chi\,(\md  q)}\overline{\chi(b)}
\sum_{\substack{\rho_{\chi} \\ |\gamma_{\chi}|\leq T}}
\frac{(x-\ell)^{\rho_{\chi}}}{\rho_{\chi}}\notag\\
=&\ 
\Sigma_1-\Sigma_2, \text{ say.}\notag
\end{align}

{\em Consider $\Sigma_1$.} We have
\begin{align*}
\Sigma_{1}&=\frac{1}{\varphi(q)}\sum_{\substack{\ell\leq x \\ \ell\equiv a\,(\md  q)}}
\Lambda(\ell)\int_{\ell}^x du
=\frac{1}{\varphi(q)}\int_0^x \psi(u;q,a)du,
\end{align*}
so, using \eqref{1-8}, we obtain
\begin{align}\label{2-3}
\Sigma_1
=\frac{1}{\varphi(q)^2}\int_0^x A(u,T;q,a)du+
O\left(\frac{x^2}{T}(\log qx)^2\right).
\end{align}
Inserting the definition of $A(u,T;q,a)$ yields
\begin{equation}\label{2-4}
\Sigma_1=\frac{1}{\varphi(q)^2}\left(\frac{x^2}{2}-
\sum_{\chi\,(\md  q)}\overline{\chi(a)}
\sum_{\substack{\rho_{\chi} \\ |\gamma_{\chi}|\leq T}}
\frac{x^{\rho_{\chi}+1}}{\rho_{\chi}(\rho_{\chi}+1)}\right)
+O\left(\frac{x^2}{T}(\log qx)^2\right).
\end{equation}

{\it Next consider $\Sigma_2$.}
We have
\begin{align}\label{2-5}
\Sigma_2=\frac{1}{\varphi(q)}\sum_{\chi\,(\md  q)}\overline{\chi(b)}
\sum_{\substack{\rho_{\chi} \\ |\gamma_{\chi}|\leq T}}
\Psi(\rho_{\chi},x;q,a),
\end{align}
where
\[
\Psi(\rho_{\chi},x;q,a)=\frac{1}{\rho_{\chi}}
\sum_{\substack{\ell\leq x \\ \ell\equiv a\,(\md  q)}}
\Lambda(\ell)(x-\ell)^{\rho_{\chi}}.
\]
Again using \eqref{1-8},
\begin{align*}
\Psi(\rho_{\chi},x;q,a)
&=\frac{1}{\rho_{\chi}}
\int_0^x (x-u)^{\rho_{\chi}}d\psi(u;q,a)\\
&=\frac{1}{\rho_{\chi}}
\int_0^x (x-u)^{\rho_{\chi}}\left(\frac{1}{\varphi(q)}dA(u,T;q,a)+dB(u,T;q,a)\right)\\
&=\frac{1}{\varphi(q)\rho_{\chi}}\int_0^x (x-u)^{\rho_{\chi}}du\\
&\qquad-\frac{1}{\varphi(q)\rho_{\chi}}
\sum_{\chi'\,(\md  q)}\overline{\chi'(a)}
\sum_{\substack{\rho_{\chi'} \\ |\gamma_{\chi'}|\leq T}}\int_0^x (x-u)^{\rho_{\chi}}
u^{\rho_{\chi'}-1}du\\
&\qquad\qquad+\frac{1}{\rho_{\chi}}
\int_0^x (x-u)^{\rho_{\chi}}dB(u,T;q,a)\\
&=J_1-J_2+J_3, \text{ say.}
\end{align*}
Obviously
\[
J_1
=\frac{x^{\rho_{\chi}+1}}{\varphi(q)\rho_{\chi}(\rho_{\chi}+1)}.
\]
Since
\begin{align*}
\int_0^x (x-u)^{\rho_{\chi}}
u^{\rho_{\chi'}-1}du
=x^{\rho_{\chi}+\rho_{\chi'}}\frac{\rho_{\chi}\Gamma(\rho_{\chi})\Gamma(\rho_{\chi'})}
{(\rho_{\chi}+\rho_{\chi'})\Gamma(\rho_{\chi}+\rho_{\chi'})},
\end{align*}
we have
\[
J_2=\frac{1}{\varphi(q)}
\sum_{\chi'\,(\md  q)}\overline{\chi'(a)}
\sum_{\substack{\rho_{\chi'} \\ |\gamma_{\chi'}|\leq T}}
\mathcal{Z}(\rho_{\chi},\rho_{\chi'})x^{\rho_{\chi}+\rho_{\chi'}},
\]
where
\begin{equation}
\label{EQ:Z_definition}
\mathcal{Z}(\rho_{\chi},\rho_{\chi'})=
\frac{\Gamma(\rho_{\chi})\Gamma(\rho_{\chi'})}
{(\rho_{\chi}+\rho_{\chi'})\Gamma(\rho_{\chi}+\rho_{\chi'})}
=\frac{\Gamma(\rho_{\chi})\Gamma(\rho_{\chi'})}
{\Gamma(1+\rho_{\chi}+\rho_{\chi'})}.
\end{equation}
Lastly,
\begin{align*}
J_3
&=\frac{1}{\rho_{\chi}}\Bigl[(x-u)^{\rho_{\chi}}B(u,T;q,a)\Bigr]_{u=0}^x
+\int_0^x (x-u)^{\rho_{\chi}-1}B(u,T;q,a)du\\
&=O\left(\frac{1}{|\rho_{\chi}|}\frac{x^2}{T}(\log qx)^2\right)
+\int_0^x (x-u)^{\rho_{\chi}-1}B(u,T;q,a)du.
\end{align*}
Therefore we now obtain
\begin{align}\label{2-6}
\Psi(\rho_{\chi},x;q,a)
&=\frac{x^{\rho_{\chi}+1}}{\varphi(q)\rho_{\chi}(\rho_{\chi}+1)}
-\frac{1}{\varphi(q)}
\sum_{\chi'\,(\md  q)}\overline{\chi'(a)}
\sum_{\substack{\rho_{\chi'} \\ |\gamma_{\chi'}|\leq T}}
\mathcal{Z}(\rho_{\chi},\rho_{\chi'})x^{\rho_{\chi}+\rho_{\chi'}}\notag\\
&\qquad\qquad+\int_0^x u^{\rho_{\chi}-1}B(x-u,T;q,a)du
+O\left(\frac{1}{|\rho_{\chi}|}\frac{x^2}{T}(\log qx)^2\right).\notag
\end{align}
Substituting this into \eqref{2-5} and using Lemma \ref{Lem:zero_sum}, we obtain
\begin{equation}\label{2-8}
\Sigma_2
=\frac{1}{\varphi(q)^2}\sum_{\chi\,(\md  q)}\overline{\chi(b)}
\sum_{\substack{\rho_{\chi} \\ |\gamma_{\chi}|\leq T}}
\frac{x^{\rho_{\chi}+1}}{\rho_{\chi}(\rho_{\chi}+1)}
-\Sigma_3+\Sigma_4
+O\left(\frac{x^2}{T}(\log qx)^4\right).
\end{equation}
where
\begin{align*}
\Sigma_3&=\frac{1}{\varphi(q)^2}\sum_{\chi\,(\md q)}\overline{\chi(b)}
\sum_{\chi'\,(\md q)}\overline{\chi'(a)}
\sum_{\substack{\rho_{\chi} \\ |\gamma_{\chi}|\leq T}}
\sum_{\substack{\rho_{\chi'} \\ |\gamma_{\chi'}|\leq T}}
\mathcal{Z}(\rho_{\chi},\rho_{\chi'})x^{\rho_{\chi}+\rho_{\chi'}},\\
\Sigma_4&=\frac{1}{\varphi(q)}\sum_{\chi\,(\md  q)}\overline{\chi(b)}
\sum_{\substack{\rho_{\chi} \\ |\gamma_{\chi}|\leq T}}
\int_0^x u^{\rho_{\chi}-1}B(x-u,T;q,a)du.
\end{align*}
Combining \eqref{1-9} with \eqref{2-1} yields
\[
S(x;q,a,b)=\Sigma_{1}-\Sigma_{2}+O\left(\frac{x^2}{T}(\log qx)^2\right),
\]
so with \eqref{2-4} and \eqref{2-8} we now arrive at
\begin{multline}
S(x;q,a,b)
=\frac{x^2}{2\varphi(q)^2}
-\frac{1}{\varphi(q)^2}\sum_{\chi\,(\md  q)}(\overline{\chi(a)}+\overline{\chi(b)})
\sum_{\substack{\rho_{\chi} \\ |\gamma_{\chi}|\leq T}}
\frac{x^{\rho_{\chi}+1}}{\rho_{\chi}(\rho_{\chi}+1)}\notag\\
+\Sigma_3-\Sigma_4+O\left(\frac{x^2}{T}(\log qx)^4\right).\notag
\end{multline}
We next extend the sum over zeros.
By Lemma \ref{Lem:zero_sum2}, we have
\begin{equation}\label{zero_extended}
\sum_{\substack{\rho_{\chi} \\ |\gamma_{\chi}|> T}}
\frac{x^{\rho_{\chi}+1}}{\rho_{\chi}(\rho_{\chi}+1)}\
\ll
x^2\sum_{\substack{\rho_{\chi} \\ |\gamma_{\chi}|> T}}\frac{1}{|\gamma_\chi|^2}
\ll
\frac{x^2}{T}(\log qx).
\end{equation}

Therefore we can extend the sum over zeros as

\begin{multline}\label{2-9}
S(x;q,a,b)
=\frac{x^2}{2\varphi(q)^2}
-\frac{1}{\varphi(q)^2}\sum_{\chi\,(\md  q)}(\overline{\chi(a)}+\overline{\chi(b)})
\sum_{\rho_{\chi}}
\frac{x^{\rho_{\chi}+1}}{\rho_{\chi}(\rho_{\chi}+1)}\\
+\Sigma_3-\Sigma_4+O\left(\frac{x^2}{T}(\log qx)^4\right).
\end{multline}

{\it Step 3} : The estimation of $\Sigma_3$ and $\Sigma_4$.

Lastly we estimate the remaining error terms $\Sigma_3$ and $\Sigma_4$.

{\em First consider $\Sigma_4$.}  
 The contribution of the integral on the interval $0\leq u\leq 3$ is
\begin{align*}
&\ll
\frac{1}{\varphi(q)}\frac{x}{T}(\log qx)^2
\sum_{\chi\,(\md q)}
\sum_{\substack{\rho_{\chi} \\ |\gamma_{\chi}|\leq T}}\int_0^3u^{\beta_\chi-1}du\\
&\ll
\frac{1}{\varphi(q)}\frac{x}{T}(\log qx)^2
\sum_{\chi\,(\md q)}
\sum_{\substack{\rho_{\chi} \\ |\gamma_{\chi}|\leq T}}\frac{1}{\beta_\chi}\\
&\ll
\frac{1}{\varphi(q)}\frac{x}{T}(\log qx)^2
\sum_{\chi\,(\md q)}
\sum_{\substack{\rho_{\chi}\neq 1-\beta_1\\ |\gamma_{\chi}|\leq T}}\frac{1}{\beta_\chi}
+
\frac{q^{1/2}}{\varphi(q)}\frac{x}{T}(\log qx)^4\\
&\ll
\frac{1}{\varphi(q)}\frac{x}{T}(\log qx)^2
\sum_{\chi\,(\md q)}
\sum_{\substack{\rho_{\chi}\\ |\gamma_{\chi}|\leq T\\\beta_\chi>1-c_0/\log qT}}\frac{1}{\beta_\chi}
+
\frac{x}{T}(\log qx)^4\\
&\ll
\frac{1}{\varphi(q)}\frac{x}{T}(\log qx)^3
\sum_{\chi\,(\md q)}
\sum_{\substack{\rho_{\chi} \\ |\gamma_{\chi}|\leq T}}1
+
\frac{x}{T}(\log qx)^4
\ll x(\log qx)^4
\end{align*}
by \eqref{EQ:Landau_Page_bound} provided $T\le x$, so we have
\begin{multline}\label{3-1}
\Sigma_4
=\frac{1}{\varphi(q)}\sum_{\chi\,(\md  q)}\overline{\chi(b)}\int_3^x
\Biggl(\sum_{\substack{\rho_{\chi} \\ |\gamma_{\chi}|\leq T}}
 u^{\rho_{\chi}-1}\Biggr)B(x-u,T;q,a)du\\
+O(x(\log qx)^4).
\end{multline}
If $3\leq u\leq x$
and $T\le x$, then Proposition \ref{LandauGonek}, proven in the next section,  yields
\[
\sum_{\substack{\rho_{\chi} \\ |\gamma_{\chi}|\leq T}}
u^{\rho_{\chi}}
\ll u\log(quT)\log\log u+ T\log u
\ll u(\log qx)^2+x(\log qx).
\]
Note that Proposition \ref{LandauGonek} is stated for primitive $\chi$, but the above
estimate is valid for any $\chi$.
Using this estimate
from \eqref{3-1} we obtain
\begin{equation}\label{3-2}
\Sigma_4\ll\frac{x^2}{T}(\log qx)^4
\end{equation}
if $T\le x$. 

{\it Next consider $\Sigma_3$.}
The following argument is inspired by \cite[Corrigendum]{Gra07},
but in our case we have to treat the zeros near the real line more carefully
since $\Gamma(s)$ has a simple pole at $s=0$.
We evaluate $\mathcal{Z}(\rho_{\chi},\rho_{\chi'})$ defined in \eqref{EQ:Z_definition}
for $|\gamma_{\chi}|\leq |\gamma_{\chi'}|$ by using Stirling's formula
\[
\Gamma(s)\ll(|t|+1)^{\sigma-1/2}e^{-(\pi/2)|t|},\quad
s=\sigma+it,\ 0\le\sigma\le 3,\ |t|\ge1.
\]
If $|\gamma_{\chi}|\leq|\gamma_{\chi'}|\leq 1$, then 
$|\Gamma(1+\rho_{\chi}+\rho_{\chi'})|\asymp 1$, and hence
\[
\mathcal{Z}(\rho_{\chi},\rho_{\chi'})
\ll|\rho_\chi|^{-1}|\rho_{\chi'}|^{-1}
\ll T^{1/2}|\rho_\chi|^{-1}|\rho_{\chi'}|^{-1}.
\]
If $|\gamma_{\chi}|\leq 1\leq|\gamma_{\chi'}|$,
then applying Stirling's formula to
$\Gamma(\rho_{\chi'})$ and $\Gamma(1+\rho_{\chi}+\rho_{\chi'})$,
\begin{align*}
\mathcal{Z}(\rho_{\chi},\rho_{\chi'})&\ll
|\rho_\chi|^{-1}\frac{|\gamma_{\chi'}|^{\beta_{\chi'}-1/2}
e^{-(\pi/2)|\gamma_{\chi'}|}}{(|\gamma_{\chi}+\gamma_{\chi'}|+1)
^{\beta_{\chi}+\beta_{\chi'}+1/2}
e^{-(\pi/2)|\gamma_{\chi}+\gamma_{\chi'}|}}\\
&\ll |\rho_\chi|^{-1}|\gamma_{\chi'}|^{-\beta_{\chi}-1}
\ll |\rho_\chi|^{-1}|\rho_{\chi'}|^{-1}\ll T^{1/2}|\rho_\chi|^{-1}|\rho_{\chi'}|^{-1}
\end{align*}
as in the case $|\gamma_{\chi}|\leq|\gamma_{\chi'}|\leq 1$.
If $1\leq |\gamma_{\chi}|\leq |\gamma_{\chi'}|\le T$, we have
\[
\mathcal{Z}(\rho_{\chi},\rho_{\chi'})\ll
\frac{|\gamma_{\chi}|^{\beta_{\chi}-1/2}
e^{-(\pi/2)|\gamma_{\chi}|}|\gamma_{\chi'}|^{\beta_{\chi'}-1/2}
e^{-(\pi/2)|\gamma_{\chi'}|}}
{(|\gamma_{\chi}+\gamma_{\chi'}|+1)^{\beta_{\chi}+\beta_{\chi'}+1/2}
e^{-(\pi/2)|\gamma_{\chi}+\gamma_{\chi'}|}}.
\]
When $\gamma_{\chi}$ and $\gamma_{\chi'}$ have the same sign, then the exponential
factors are cancelled and we obtain
\begin{equation}
\label{estimate-mathcalZ}
\mathcal{Z}(\rho_{\chi},\rho_{\chi'})
\ll
|\gamma_{\chi}|^{\beta_{\chi}-1/2}
|\gamma_{\chi'}|^{-\beta_{\chi}-1}
\ll
|\gamma_{\chi}|^{-1/2}
|\gamma_{\chi'}|^{-1}
\ll
T^{1/2}
|\gamma_{\chi}|^{-1}
|\gamma_{\chi'}|^{-1}
\end{equation}
since $|\gamma_\chi|\le|\gamma_{\chi'}|\le T$.
When they have opposite signs the contribution of the exponential factors is
$O(e^{-\pi|\gamma_{\chi}|})$, and
\begin{align*}
&(|\gamma_{\chi}+\gamma_{\chi'}|+1)^{-(\beta_{\chi}+\beta_{\chi'}+1/2)}\\
{}={}&
(1+|\gamma_{\chi'}|-|\gamma_{\chi}|)^{-(\beta_{\chi}+\beta_{\chi'}+1/2)}\\
{}={}&
(1+|\gamma_{\chi'}|)^{-(\beta_{\chi}+\beta_{\chi'}+1/2)}
\left(1-\frac{|\gamma_{\chi}|}{1+|\gamma_{\chi'}|}\right)^{-(\beta_{\chi}+\beta_{\chi'}+1/2)}\\
{}\le{}&
(1+|\gamma_{\chi'}|)^{-(\beta_{\chi}+\beta_{\chi'}+1/2)}
(1+|\gamma_{\chi}|)^{(\beta_{\chi}+\beta_{\chi'}+1/2)}\\
{}\le{}&
(1+|\gamma_{\chi'}|)^{-(\beta_{\chi}+\beta_{\chi'}+1/2)}
(1+|\gamma_{\chi}|)^{\pi}\\
{}\le{}&
(1+|\gamma_{\chi'}|)^{-(\beta_{\chi}+\beta_{\chi'}+1/2)}e^{\pi|\gamma_{\chi}|},
\end{align*}
we again obtain
$\mathcal{Z}(\rho_{\chi},\rho_{\chi'})\ll T^{1/2}|\gamma_{\chi}|^{-1}|\gamma_{\chi'}|^{-1}$.
Therefore, the estimate
\begin{equation}
\label{estimate-mathcalZ_unified}
\mathcal{Z}(\rho_{\chi},\rho_{\chi'})\ll T^{1/2}|\rho_{\chi}|^{-1}|\rho_{\chi'}|^{-1}
\end{equation}
holds for all $|\gamma_{\chi}|,|\gamma_{\chi'}|\le T$
by symmetry between $\rho_\chi$ and $\rho_{\chi'}$.

By using the estimate \eqref{estimate-mathcalZ_unified}, we have
\[
\sum_{\substack{\rho_{\chi} \\[1.4pt] |\gamma_{\chi}|\leq T}}
\sum_{\substack{\rho_{\chi'} \\ |\gamma_{\chi'}|\leq T}}
\mathcal{Z}(\rho_\chi,\rho_{\chi'})x^{\rho_{\chi}+\rho_{\chi'}}
{}\ll{}
x^{2B_q}T^{1/2}
\sum_{\substack{\rho_{\chi} \\[1.4pt] |\gamma_{\chi}|\leq T}}
\sum_{\substack{\rho_{\chi'} \\ |\gamma_{\chi'}|\leq T}}
|\rho_{\chi}|^{-1}|\rho_{\chi'}|^{-1}.
\]
Therefore, by Lemma \ref{Lem:zero_sum}, we have
\begin{align*}
\Sigma_3
&\ll
x^{2B_q}T^{1/2}
\left(\frac{1}{\varphi(q)}\sum_{\chi\,(\md q)}
\sum_{\substack{\rho_\chi\\|\gamma_\chi|\le T}}\frac{1}{|\rho_\chi|}\right)^2\\
&\ll
x^{2B_q}T^{1/2}\left((\log qx)^2+\frac{q^{1/2}(\log q)^2}{\varphi(q)}\right)^2
\ll
x^{2B_q}T^{1/2}(\log qx)^4
\end{align*}
if $T\le x$.

All the error terms on the right-hand side of \eqref{2-9} have now been estimated.
We choose the optimal $T$ by requiring that
$x^{2B_q}T^{1/2}=x^2/T$, hence $T=x^{4(1-B_q)/3}$.   Since $B_q\geq 1/2$  this choice 
satisfies the condition $T\le x$.    Substituting this choice of $T$ into \eqref{2-9}, 
we obtain the assertion of Proposition \ref{prop1}.

\section{The Landau--Gonek formula for $L$-functions}\label{Landau}
The Landau--Gonek result \cite{Gon93}, originally a formula on the zeros of $\zeta(s)$, 
has been extended to other situations, for example Ford et al.\ \cite{FSZ} worked on a general setting of the Selberg class. We did not find any instance where the uniformity with respect to $q$ was treated and we do so here for the sake of completeness. 

\begin{prop} \label{LandauGonek}
Let $x,T,q>1$ and $\chi$ be a primitive 
character $(\md q)$.
Then
\begin{align*}
\sum_{\substack{\rho_{\chi} \\ |\gamma_{\chi}|\leq T}}x^{\rho_{\chi}}
&=-\frac{T}{\pi}\chi(x)
\Lambda(x)+O(x(\log2qxT)(\log\log3x))\\
&\qquad+O\left((\log x) \min\left\{T,\frac{x}{\langle x\rangle}\right\}\right)
+O\left((\log2qT)\min\left\{T,\frac{1}{\log x}\right\}\right),
\end{align*}
where $\chi(x)=\Lambda(x)=0$ if $x$ is not an integer, and $\langle x\rangle$ denotes the distance
from $x$ to the nearest prime power other than $x$ itself.
\end{prop}

\begin{proof}
The proof essentially follows the original one of Gonek \cite{Gon93}.
Let $c=1+1/\log3x$, and consider the integral 
\begin{align}\label{4-1}
I
&=\frac{1}{2\pi i}
\left(\int_{c-iT}^{c+iT}+\int_{c+iT}^{1-c+iT}+\int_{1-c+iT}^{1-c-iT}+\int_{1-c-iT}^{c-iT}
\right)\frac{L'}{L}(s,\chi)x^s ds\\
&=I_1+I_2+I_3+I_4,\notag
\end{align}
say.
First suppose that the horizontal paths do not cross any zero of $L(s,\chi)$.
The poles inside the rectangle are the non-trivial zeros of $L(s,\chi)$
and at most one  trivial zero of $L(s,\chi)$ at $s=0$.
Hence the residue theorem gives
\begin{align}\label{4-2}
I=\sum_{\substack{\rho_{\chi} \\ |\gamma_{\chi}|\leq T}}
x^{\rho_{\chi}}+O(1).
\end{align}

We evaluate $I_1, I_2, I_3$ and $I_4$.
First consider $I_2$. We use the well-known formula
\begin{align}\label{4-3}
\frac{L'}{L}(s,\chi)=\sum_{\substack{\rho_{\chi} \\ |\gamma_{\chi}-t|\leq 1}}
\frac{1}{s-\rho_{\chi}}
+O(\log q(|t|+2))
\end{align}
uniformly in $-1\leq\sigma\leq 2$ (\cite[Lemma 12.6]{MV}).
We have
\begin{align}\label{4-4}
I_2=\sum_{\substack{\rho_{\chi} \\ |\gamma_{\chi}-T|\leq 1}}\int_{c+iT}^{1-c+iT}
\frac{x^s}{s-\rho_{\chi}}ds+O\left(\log2qT\int_{1-c}^{c}x^{\sigma}d\sigma\right),
\end{align} 
whose error term is
\begin{align}\label{4-5}
\ll x^c\log2qT\ll x\log2qT.
\end{align}
For each integral on the first sum of \eqref{4-4}, we first observe that
$T-1\leq\gamma_{\chi}\leq T+1$.
When $T-1\leq\gamma_{\chi}\leq T$, we deform the path of integration as
\[
\int_{c+iT}^{1-c+iT}=\int_{c+iT}^{c+i(T+1)}+\int_{c+i(T+1)}^{1-c+i(T+1)}
+\int_{1-c+i(T+1)}^{1-c+iT}.
\]
Noting that the denominator on the second term is $\gg 1$ and that $\log3x\gg 1$,
$\log\log3x\gg 1$,
we obtain in this case
\begin{align*}
\int_{c+iT}^{1-c+iT} \frac{x^{s}}{s-\rho_{\chi}}ds
&\ll x^c\int_T^{T+1}\frac{dt}{|(c-\beta_{\chi})+i(t-\gamma_{\chi})|}
+\int_{1-c}^c x^{\sigma}d\sigma +\frac{x^{1-c}}{\beta_{\chi}-(1-c)}\\
&\ll x\int_{\gamma_{\chi}}^{\gamma_{\chi}+2}\min\left\{\log3x,\frac{1}{t-\gamma_{\chi}}
\right\}dt +x+\log3x\\
&\ll x\left(\int_{\gamma_{\chi}}^{\gamma_{\chi}+1/\log3x}\log3x \,dt
+\int_{\gamma_{\chi}+1/\log3x}^{\gamma_{\chi}+2}\frac{dt}{t-\gamma_{\chi}}\right)
+x\\
&\ll x\left(\frac{1}{\log3x}\cdot\log3x
+\Bigl[\log(t-\gamma_{\chi})\Bigr]_{t=\gamma_{\chi}+1/\log3x}^{\gamma_{\chi}+2}\right)+x
\\ &\ll x\log\log3x.
\end{align*}
If $T<\gamma_{\chi}\leq T+1$, we deform the path to that including the segment with the
imaginary part $T-1$, and argue similarly.   
We can conclude that
\[
I_2\ll
x\log\log3x\sum_{\substack{\rho_{\chi} \\ |\gamma_{\chi}-T|\leq 1}}1
+x\log2qT
\]
and hence, by using Lemma \ref{Lem:zero_counting_local},
\begin{align}\label{4-6}
I_2\ll x(\log2qxT)(\log\log3x).
\end{align}
The estimate of $I_4$ is similar.

Next consider $I_3$. We first quote
\begin{align}\label{4-7}
\frac{L'}{L}(s,\chi)=-\frac{L'}{L}(1-s,\overline{\chi})-\log\frac{q}{2\pi}
-\frac{\Gamma'}{\Gamma}(1-s)+\frac{\pi}{2}\cot\frac{\pi}{2}(s+\kappa)
\end{align}
(\cite[(10.35)]{MV}), where $\kappa=0$ or 1 depending on whether $\chi$ is an even or an odd character,
respectively.   We see easily that
\[
\frac{\pi}{2}\cot\frac{\pi}{2}(s+\kappa)=\pm i+O(e^{-\pi |t|})
\]
for $|t|\ge1$ and $\Re s=1-c$.
Therefore, putting $s=1-c+it$ and applying Stirling's formula we have
\begin{align}\label{4-8}
\frac{L'}{L}(1-c+it,\chi)=-\frac{L'}{L}(c-it,\overline{\chi})-\log qt+C+O(t^{-1})
\end{align}
for $t\geq 1$, where $C$ denotes a constant.
Hence, the part $[1-c\pm i,1-c\pm iT]$ of the integral $I_3$ is
\[
=\frac{\pm1}{2\pi}
\int_1^T \left(\frac{L'}{L}(c\mp it,\overline{\chi})+\log qt-C\right)x^{1-c\pm it}dt
+O\left(\int_1^T\frac{x^{1-c}}{t}dt\right),
\]
whose error term is $O(\log T)$. The integral is
\begin{align*}
= \mp\frac{x^{1-c}}{2\pi}\sum_{n=2}^{\infty}\frac{\overline{\chi(n)}\Lambda(n)}{n^c}\int_1^T(nx)^{\pm it}dt
\pm\frac{x^{1-c}}{2\pi}\int_1^T(\log qt -C)x^{\pm it}dt,
\end{align*}
whose first part is
\[
\ll x^{1-c}\sum_{n=2}^{\infty}\frac{\Lambda(n)}{n^c \log nx}
\ll \sum_{n=2}^{\infty}\frac{\Lambda(n)}{n^c}
=-\frac{\zeta'}{\zeta}(c)\ll\frac{1}{c-1}\ll\log3x.
\]
The second part is trivially $O(T\log2qT)$, while integration by parts gives
\[
=\frac{x^{1-c}}{2\pi}\left(\left[(\log qt-C)\frac{x^{\pm it}}{i\log x}\right]_{t=1}^T
-\int_1^T\frac{x^{it}}{it \log x}dt\right)
\ll\frac{\log2qT}{\log x}.
\]
(Note that $\log x\gg1$ does not hold.)
The part $[1-c-i,1-c+i]$ of the integral $I_3$
is $\ll\log2qx$ by \eqref{4-7}.
Therefore we conclude
\begin{equation}
\label{4-9}
\begin{aligned}
I_3
&\ll\log2qx+(\log2qT)\min\left\{T,\frac{1}{\log x}\right\}+\log T\\
&\ll x(\log2qT)(\log\log3x)+(\log2qT)\min\left\{T,\frac{1}{\log x}\right\}
\end{aligned}
\end{equation}
since $T>1$.

We next consider $I_1$. Substituting the Dirichlet series expansion, we have
\begin{align}\label{4-10}
I_1
&=-\sum_{n=2}^{\infty}\chi(n)\Lambda(n)\frac{1}{2\pi}\int_{-T}^T (x/n)^{c+it}dt\\
&=-\frac{T}{\pi}\chi(x)\Lambda(x)
+O\left(\sum_{n\neq x}\Lambda(n)(x/n)^c \min\left\{T,\frac{1}{|\log x/n|}\right\}\right).\notag
\end{align}
The error term here can be estimated by \cite[Lemma 2]{Gon93}, and so
\begin{align}\label{4-11}
I_1=-\frac{T}{\pi}\chi(x)\Lambda(x)+O(x(\log2x)(\log\log3x))
+O\left((\log x) \min\left\{T,\frac{x}{\langle x\rangle}\right\}\right).
\end{align} 
The formula of the lemma follows by combining \eqref{4-6}, \eqref{4-9} and \eqref{4-11}.

Finally if the path of $I_2$ or $I_4$ crosses some zero we choose $T'$ slightly larger than
$T$, and define $I'$, similar to $I$ but now $T$ is replaced by $T'$.    Then instead of
\eqref{4-2} we obtain
\[
I'=\sum_{\substack{\rho_{\chi} \\ |\gamma_{\chi}|\leq T}}
x^{\rho_{\chi}}+O(x\log2qT)
\]
(by Lemma \ref{Lem:zero_counting_local}), while the evaluation of integrals on the edges of the rectangle
can be done in the same way as for $I$, so the assertion of the lemma is also valid in this case.
\end{proof}

\section{Lemmas for the proof of Theorem \ref{Asymptotic_B_SP}}\label{Lemma_B_SP}
In this section we present some preparatory material  for the improvement of the error term of Proposition \ref{prop1}.
We start with a lemma on an integral of the Selberg-type. In order to prove this first lemma,
we need to calculate the following sum over the non-trivial zeros of $L(s,\chi)$.

\begin{lem}
\label{Lem:offdiagonal_zero}
For any $x\ge2$ and real number $y$, we have
\[
\sum_{\substack{\rho_\chi\\|\gamma_\chi|\le x}}\frac{1}{1+|\gamma_\chi-y|}\ll(\log qx)^2.
\]
\end{lem}
\begin{proof}
If $|y|>2x$, then each term above is $\ll x^{-1}$ and the lemma holds trivially. So we consider the case $|y|\le 2x$.
Then by the triangle inequality, we have
\[|\gamma_\chi-y|\le x+y\le3x.\]
Thus we can extend the sum and dissect it as
\begin{align*}
\sum_{\substack{\rho_\chi\\|\gamma_\chi|\le x}}\frac{1}{1+|\gamma_\chi-y|}
&\le
\sum_{\substack{\rho_\chi\\|\gamma_\chi-y|\le 3x}}\frac{1}{1+|\gamma_\chi-y|}\\
&\le
\sum_{\substack{\rho_\chi\\|\gamma_\chi-y|\le 1}}\frac{1}{1+|\gamma_\chi-y|}
+
\sum_{1\le n\le 3x}\sum_{\substack{\rho_\chi\\n<|\gamma_\chi-y|\le n+1}}\frac{1}{1+|\gamma_\chi-y|}\\
&\le
\sum_{\substack{\rho_\chi\\|\gamma_\chi-y|\le 1}}1+\sum_{1\le n\le 3x}\frac{1}{n}\sum_{\substack{\rho_\chi\\n<|\gamma_\chi-y|\le n+1}}1.
\end{align*}
By using Lemma \ref{Lem:zero_counting_local}, we can estimate the last sum to be
\[
\ll
(\log qx)\left(1+\sum_{1\le n\le 3x}\frac{1}{n}\right)
\ll
(\log qx)(\log x)
\ll
(\log qx)^2.
\]

\end{proof}

We now obtain an estimate for an integral of the Selberg-type.

\begin{lem}
\label{Selberg}
For any $2\le h\le x$ and any $\chi\ (\md q)$, we have
\[
\int_{x}^{2x}\Big|\sum_{t< n\le t+h}\chi(n)\Lambda(n)-\delta_0(\chi)h\Big|^2dt
\ll
hx^{2B_q^\ast}(\log qx)^4.
\]
\end{lem}

\begin{proof}
By taking the difference between $u=t+h$ and $u=t$ in Lemma \ref{Lem:explicit}, we have
\[
\sum_{t< n\le t+h}\chi(n)\Lambda(n)
=
\delta_0(\chi)h-\sum_{\substack{\rho_\chi\\|\gamma_\chi|\le T}}\frac{(t+h)^{\rho_\chi}-t^{\rho_\chi}}{\rho_\chi}
+E(t+h,T,\chi)-E(t,T,\chi),
\]
where the term $C(\chi^\ast)$ is cancelled out since it is independent of the main variable $u$ in Lemma \ref{Lem:explicit}.
We take the parameter $T=x$.
Then the error term is estimated by
\[
E(t+h,T,\chi)-E(t,T,\chi)\ll(\log 2q)(\log x)+\frac{x}{T}(\log qxT)^2\ll(\log qx)^2.
\]
Therefore, we obtain
\begin{align*}
&\sum_{t< n\le t+h}\chi(n)\Lambda(n)-\delta_0(\chi)h\\
{}={}&
-\sum_{\substack{\rho_\chi\\|\gamma_\chi|\le x}}\frac{(t+h)^{\rho_\chi}-t^{\rho_\chi}}{\rho_\chi}+O((\log qx)^2)\\
{}={}&
-\sum_{\substack{\rho_\chi\\|\gamma_\chi|\le x/h}}\int_{t}^{t+h}u^{\rho_\chi-1}du
-\sum_{\substack{\rho_\chi\\x/h<|\gamma_\chi|\le x}}\frac{(t+h)^{\rho_\chi}-t^{\rho_\chi}}{\rho_\chi}+O((\log qx)^2)\\
{}={}&
-\int_{t}^{t+h}\Biggl(\sum_{\substack{\rho_\chi\\|\gamma_\chi|\le x/h}}
u^{\rho_\chi-1}\Biggr)du
-\sum_{\substack{\rho_\chi\\x/h<|\gamma_\chi|\le x}}
\frac{(t+h)^{\rho_\chi}-t^{\rho_\chi}}{\rho_\chi}
+O((\log qx)^2)\\
{}={}&
-\psi_{1}(t)
-\psi_{2}(t)
+O\left((\log qx)^2\right),\text{ say.}
\end{align*}
Substituting this explicit formula into the integral of the assertion, we have
\begin{gather*}
\int_{x}^{2x}
\Big|\sum_{t< n\le t+h}\chi(n)\Lambda(n)-\delta_0(\chi)h\Big|^2dt\\
\ll
\int_{x}^{2x}|\psi_1(t)|^2dt
+
\int_{x}^{2x}|\psi_2(t)|^2dt
+
x(\log qx)^4.
\end{gather*}
Hence it suffices to prove the two estimates
\[
\int_{x}^{2x}|\psi_1(t)|^2dt,\quad \int_{x}^{2x}|\psi_2(t)|^2dt
\ll
hx^{2B_q^\ast}(\log qx)^4,
\]
since $B_q^\ast\ge1/2$ implies $x(\log qx)^4\le hx^{2B_q^\ast}(\log qx)^4$.

Applying the Cauchy--Schwarz inequality to the first integral we have
\[
|\psi_1(t)|^2
\ll
h\int_{t}^{t+h}
\Big|\sum_{\substack{\rho_\chi\\|\gamma_\chi|\le x/h}}u^{\rho_\chi-1}\Big|^2du
\]
so that
\begin{align*}
\int_{x}^{2x}|\psi_1(t)|^2dt
&\ll
h\int_{x}^{2x}\int_{t}^{t+h}
\Big|\sum_{\substack{\rho_\chi\\|\gamma_\chi|\le x/h}}u^{\rho_\chi-1}\Big|^2dudt\\
&=
h\int_{x}^{2x+h}\Big|\sum_{\substack{\rho_\chi\\|\gamma_\chi|\le x/h}}u^{\rho_\chi-1}\Big|^2\left(\int_{\max(u-h,x)}^{\min(u,2x)}dt\right)du\\
&\ll
h^2\int_x^{3x}
\Big|\sum_{\substack{\rho_\chi\\|\gamma_\chi|\le x/h}}u^{\rho_\chi-1}\Big|^2du
\ll
h^2x^{-2}\int_x^{3x}
\Big|\sum_{\substack{\rho_\chi\\|\gamma_\chi|\le x/h}}u^{\rho_\chi}\Big|^2du
\end{align*}
since $h\le x$.
We now expand the square and integrate over $u$.
This gives
\[
\int_{x}^{2x}|\psi_1(t)|^2dt
\ll
h^2x^{-2}
\dsum_{\substack{\rho_\chi,\rho_\chi'\\|\gamma_\chi|,|\gamma'_\chi|\le x/h}}
\frac{x^{\beta_\chi+\beta'_\chi+1}}{1+|\gamma_\chi-\gamma'_\chi|}.
\]
Using the Vinogradov-Korobov zero-free region
(For a detailed proof of a weaker result, see \cite[Theorem 1]{Tatuzawa}.) and Siegel's theorem,
we can estimate $\beta_\chi$ and $\beta_{\chi'}$ by $B_q^\ast$.
Therefore,
\begin{align*}
\int_{x}^{2x}|\psi_1(t)|^2dt
&\ll
h^2x^{2B_q^\ast-1}
\sum_{\substack{\rho_\chi\\|\gamma_\chi|\le x/h}}
\sum_{\substack{\rho_\chi'\\|\gamma_\chi'|\le x}}
\frac{1}{1+|\gamma_\chi-\gamma'_\chi|}\\
&\ll
h^2x^{2B_q^\ast-1}(\log qx)^2\sum_{\substack{\rho_\chi\\|\gamma_\chi|\le x/h}}1
\ll
hx^{2B_q^\ast}(\log qx)^3
\end{align*}
by Lemma \ref{Lem:offdiagonal_zero}.

For the latter integral, we ignore the difference as
\begin{align*}
\int_{x}^{2x}|\psi_2(t)|^2dt
&\ll
\int_{x}^{2x}
\Big|\sum_{\substack{\rho_\chi\\x/h<|\gamma_\chi|\le x}}\frac{(t+h)^{\rho_\chi}}{\rho_\chi}\Big|^2dt
+
\int_{x}^{2x}
\Big|\sum_{\substack{\rho_\chi\\x/h<|\gamma_\chi|\le x}}\frac{t^{\rho_\chi}}{\rho_\chi}\Big|^2dt\\
&=
\int_{x+h}^{2x+h}
\Big|\sum_{\substack{\rho_\chi\\x/h<|\gamma_\chi|\le x}}\frac{t^{\rho_\chi}}{\rho_\chi}\Big|^2dt
+
\int_{x}^{2x}
\Big|\sum_{\substack{\rho_\chi\\x/h<|\gamma_\chi|\le x}}\frac{t^{\rho_\chi}}{\rho_\chi}\Big|^2dt\\
&\ll
\int_{x}^{3x}
\Big|\sum_{\substack{\rho_\chi\\x/h<|\gamma_\chi|\le x}}\frac{t^{\rho_\chi}}{\rho_\chi}\Big|^2dt.
\end{align*}
Then we expand the square and integrate over $t$. This gives
\begin{align*}
\int_{x}^{2x}|\psi_2(t)|^2dt
\ll\!\!\!
\dsum_{\substack{\rho_\chi,\rho_\chi'\\x/h<|\gamma_\chi|,|\gamma'_\chi|\le x}}\!\!\!
\frac{x^{\beta_\chi+\beta'_\chi+1}}
{|\gamma_\chi||\gamma'_\chi|(1+|\gamma_\chi-\gamma'_\chi|)}.
\end{align*}
Obviously,
\[
\frac{x^{\beta_\chi+\beta'_\chi}}{|\gamma_\chi||\gamma'_\chi|}
\ll
\frac{x^{2\beta_\chi}}{|\gamma_\chi|^2}
+
\frac{x^{2\beta'_\chi}}{|\gamma'_\chi|^2}.
\]
By using the symmetry of the terms in $\gamma_{\chi}$ and $\gamma_{\chi'}$) we have
\begin{align*}
\int_{x}^{2x}|\psi_2(t)|^2dt
&\ll
x^{1+2B_q^\ast}
\sum_{\substack{\rho_\chi\\x/h<|\gamma_\chi|\le x}}
\frac{1}{|\gamma_\chi|^2}
\sum_{\substack{\rho_\chi'\\|\gamma'_\chi|\le x}}
\frac{1}{1+|\gamma_\chi-\gamma'_\chi|}\\
&\ll
x^{1+2B_q^\ast}(\log qx)^2
\sum_{\substack{\rho_\chi\\x/h<|\gamma_\chi|\le x}}\frac{1}{|\gamma_\chi|^2}
\ll
hx^{2B_q^\ast}(\log qx)^3,
\end{align*}
where we used Lemma \ref{Lem:zero_sum2} for the last estimate.
This completes the proof.
\end{proof}

Now let
\[
T(\alpha)=\sum_{n\le x}e(n\alpha),\quad
S(\alpha,\chi)=\sum_{n\le x}\chi(n)\Lambda(n)e(n\alpha),
\] and
\[
W(\alpha,\chi)
=
S(\alpha,\chi)-\delta_0(\chi)T(\alpha),
\]
where $e(\alpha)=\exp(2\pi i\alpha)$.
Our next task is to translate the previous estimate of the integral
into this exponential sum setting.

\begin{lem}
\label{W_PNT}
Let $x^{-1}\le\xi\le1/2$.
For any $\chi\ (\md q)$, we have
\[
\int_{-\xi}^{\xi}|W(\alpha,\chi)|^2d\alpha\ll\xi x^{2B_q^\ast}(\log qx)^4.
\]
\end{lem}
\begin{proof}
We first note that
\[
W(\alpha,\chi)
=
\sum_{0<n\le x}\left(\chi(n)\Lambda(n)-\delta_0(\chi)\right)e(n\alpha).
\]
Thus using Gallagher's lemma \cite[Lemma 1]{Gal70}, we have
\begin{align*}
\int_{-\xi}^{\xi}|W(\alpha,\chi)|^2d\alpha
&=
\int_{-\xi}^{\xi}\Bigg|\sum_{0<n\le x}\left(\chi(n)\Lambda(n)-\delta_0(\chi)\right)e(n\alpha)\Bigg|^2d\alpha.\\
&\ll
\xi^2\int_{-(2\xi)^{-1}}^{x}
\Bigg|\sum_{a(t)<n\le b(t)}\left(\chi(n)\Lambda(n)-\delta_0(\chi)\right)\Bigg|^2dt,
\end{align*}
where
\[
a(t)=\max(t,0),\quad
b(t)=\min(t+(2\xi)^{-1},x).
\]
We decompose this integral as
\[
\ll
\xi^2\int_{-(2\xi)^{-1}}^{(2\xi)^{-1}}
+
\xi^2\int_{(2\xi)^{-1}}^{x-(2\xi)^{-1}}
+
\xi^2\int_{x-(2\xi)^{-1}}^{x}
=
\xi^2I_{-}+\xi^2I+\xi^2I_{+},\text{ say}.
\]
By Lemma \ref{Lem:explicit}, the Vinogadov-Korobov zero free-region,
and Siegel's theorem we obtain
\[
\sum_{n\le t}\chi(n)\Lambda(n)-\delta_0(\chi)t
\ll
x^{B_q^\ast}\sum_{\substack{\rho_\chi\neq1-\beta_1\\|\gamma_\chi|\le x}}\frac{1}{|\rho_\chi|}
+\left|\delta_1(\chi)\frac{t^{1-\beta_1}}{1-\beta_1}-C(\chi^\ast)\right|+(\log qx)^2
\]
for $2\le t\le 2x$. The second term on the right-hand side is estimated
by using Theorem 11.4 and formula (12.7) of \cite{MV} as
\begin{align*}
\delta_1(\chi)\frac{t^{1-\beta_1}}{1-\beta_1}-C(\chi^\ast)
&=
\delta_1(\chi)\frac{t^{1-\beta_1}-1}{1-\beta_1}+O(\log2q)\\
&=
\delta_1(\chi)\frac{\log t}{1-\beta_1}\int_0^{1-\beta_1}t^{\sigma}d\sigma+O(\log 2q)\\
&\ll
x^{1-\beta_1}(\log 2x)+\log 2q
\ll
x^{1/2}(\log qx),
\end{align*}
since $\beta_1>1/2$. Thus, by Lemma \ref{Lem:zero_sum}, we have
\begin{align*}
\sum_{n\le t}\chi(n)\Lambda(n)-\delta_0(\chi)t
&\ll
x^{B_q^\ast}(\log qx)^2+x^{1/2}(\log qx)+(\log qx)^2
\ll
x^{B_q^\ast}(\log qx)^2,
\end{align*}
which also holds trivially for $0\le t<2$. Thus for any $0\le a\le b\le 2x$, we have
\begin{align*}
\sum_{a<n\le b}\left(\chi(n)\Lambda(n)-\delta_0(\chi)\right)
&=\sum_{a<n\le b}\chi(n)\Lambda(n)-\delta_0(\chi)(b-a)+O(1)\\
&\ll x^{B_q^\ast}(\log qx)^2.
\end{align*}
By substituting this estimate into $I_{\pm}$, we obtain
\[
\xi^2I_{\pm}\ll\xi x^{2B_q^\ast}(\log qx)^4,
\]
since $I_{\pm}$ are integrals taken over intervals of length $\le\xi^{-1}$.
Finally,
\begin{align*}
\xi^2I
{}={}&
\xi^2\int_{(2\xi)^{-1}}^{x-(2\xi)^{-1}}
\Bigg|\sum_{t<n\le t+(2\xi)^{-1}}(\chi(n)\Lambda(n)-\delta_0(\chi))\Bigg|^2dt\\
{}\ll{}&
\xi^2\int_{(2\xi)^{-1}}^{x}
\Bigg|\sum_{t<n\le t+(2\xi)^{-1}}\chi(n)\Lambda(n)-\delta_0(\chi)(2\xi)^{-1}\Bigg|^2dt+\xi^2x\\
{}\ll{}&
\xi^2\sum_{k=0}^{O(\log x)}
\int_{x/{2^{k+1}}}^{x/2^k}
\Bigg|\sum_{t<n\le t+(2\xi)^{-1}}\chi(n)\Lambda(n)-\delta_0(\chi)(2\xi)^{-1}\Bigg|^2dt+\xi^2 x\\
{}\ll{}&
\xi x^{2B_q^\ast}(\log qx)^4
\end{align*}
by Lemma \ref{Selberg}.
Summing up the above calculations,
we obtain the lemma.
\end{proof}

Let
\[
J(\chi)
=
\int_{-1/2}^{1/2}|W(\alpha,\chi)|^2|T(\alpha)|d\alpha.
\]
Now by using the previous results we can obtain an estimate for this quantity.
\begin{lem}
\label{J_estimate}
We have
\[
J(\chi)
\ll
x^{2B_q^\ast}(\log qx)^5.
\]
\end{lem}
\begin{proof}
We dissect the integral dyadically as
\[
J(\chi)
\le
\int_{|\alpha|\le1/x}|W(\alpha,\chi)|^2|T(\alpha)|d\alpha
+
\sum_{k=1}^{O(\log x)}
\int_{1/2^{k+1}<|\alpha|\le1/2^k}|W(\alpha,\chi)|^2|T(\alpha)|d\alpha.
\]
Then since $T(\alpha)\ll\min(x,|\alpha|^{-1})$ for $|\alpha|\le1/2$, we have
\[
J(\chi)
\ll
(\log x)
\sup_{1/x<\xi\le1/2}
\xi^{-1}\int_{|\alpha|\le\xi}|W(\alpha,\chi)|^2d\alpha
\ll
x^{2B_q^\ast}(\log qx)^5
\]
by Lemma \ref{W_PNT}. 
\end{proof}

\section{Proofs of Theorems \ref{Asymptotic_B_SP} and \ref{Asymptotic_1B}}\label{Asymp2}
We  let
\[
G(n;\chi_1,\chi_2)=\sum_{\ell+m=n}\chi_1(\ell)\Lambda(\ell)\chi_2(m)\Lambda(m),\quad
S(x;\chi_1,\chi_2)=\sum_{n\le x}G(n;\chi_1,\chi_2)
\]
and prove the following intermediate lemma.
\begin{lem}
\label{pre_lem}
For $x\ge2$ and $\chi_1,\chi_2\ (\md q)$, we have
\begin{multline*}
S(x;\chi_1,\chi_2)
=
\frac{\delta_0(\chi_1)\delta_0(\chi_2)}{2}x^2
-\delta_0(\chi_2)H(x,\chi_1)-\delta_0(\chi_1)H(x,\chi_2)
+R(x;\chi_1,\chi_2)\\
+O\Bigl(\delta_0(\chi_2)(1+\delta_1(\chi_1)q^{1/2})x(\log qx)^2
+\delta_0(\chi_1)(1+\delta_1(\chi_2)q^{1/2})x(\log qx)^2\Bigr),
\end{multline*}
where
\[
H(x,\chi)
=
\sum_{\rho_\chi}\frac{x^{\rho_\chi+1}}{\rho_\chi(\rho_\chi+1)},\quad
R(x;\chi_1,\chi_2)
=
\int_0^1W(\alpha,\chi_1)W(\alpha,\chi_2)T(-\alpha)d\alpha.
\]
\end{lem}
\begin{proof}
By the orthogonality of the exponential function we have
\begin{equation}
\label{EQ:S_integral_expression}
S(x;\chi_1,\chi_2)=\int_0^1S(\alpha,\chi_1)S(\alpha,\chi_2)T(-\alpha)d\alpha.
\end{equation}
From the definition we have an expansion
\begin{multline}
S(\alpha,\chi_1)S(\alpha,\chi_2)
=
\delta_0(\chi_2)S(\alpha,\chi_1)T(\alpha)
+
\delta_0(\chi_1)S(\alpha,\chi_2)T(\alpha)\\
-
\delta_0(\chi_1)\delta_0(\chi_2)T(\alpha)^2
+
W(\alpha,\chi_1)W(\alpha,\chi_2).
\end{multline}
Substituting this decomposition into the integral expression \eqref{EQ:S_integral_expression}, we have
\[
S(x;\chi_1,\chi_2)
=
\delta_0(\chi_2)I(\chi_1)+\delta_0(\chi_1)I(\chi_2)
-\delta_0(\chi_1)\delta_0(\chi_2)I
+R(x;\chi_1,\chi_2),
\]
where
\[
I
=
\int_0^1T(\alpha)^2T(-\alpha)d\alpha,\quad
I(\chi)
=
\int_0^1S(\alpha,\chi)T(\alpha)T(-\alpha)d\alpha.
\]
Therefore it is sufficient to show that
\begin{equation}
\label{EQ:I_goal}
I=\frac{x^2}{2}+O(x),\quad
I(\chi)=\frac{\delta_0(\chi)x^2}{2}-H(x,\chi)+O((1+\delta_1(\chi)q^{1/2})x(\log qx)^2).
\end{equation}
The first integral $I$ is evaluated by using the orthogonality as
\[
I=\sum_{\ell+m\le x}1=\sum_{n\le x}(n-1)=\frac{x^2}{2}+O(x).
\]
The second integral $I(\chi)$ is
\begin{align*}
I(\chi)
&=
\sum_{\ell+m\le x}\chi(\ell)\Lambda(\ell)
=
\sum_{n\le x}(x-n)\chi(n)\Lambda(n)+O(x)\\
&=
\sum_{n\le x}\chi(n)\Lambda(n)\int_{n}^{x}du+O(x)
=\int_0^x \psi(u,\chi)du+O(x)
\end{align*}
by partial summation. We substitute Lemma \ref{lem-1} with $T=x$.
Then
\begin{equation}
\label{EQ:pre_I_goal}
I(\chi)
=
\frac{\delta_0(\chi)x^2}{2}
-\sum_{\substack{\rho_\chi\\|\gamma_\chi|\le x}}\frac{x^{\rho_\chi+1}}{\rho_\chi(\rho_\chi+1)}
+O(x(\log qx)^2+\delta_1(\chi)q^{1/2}(\log q)^2x).
\end{equation}
We then extend the sum over zeros which this gives the error term of the size
\[
\sum_{\substack{\rho_\chi\\|\gamma_\chi|>x}}\frac{x^{\rho_\chi+1}}{\rho_\chi(\rho_\chi+1)}
\ll
x^{2}\sum_{\substack{\rho_\chi\\|\gamma_\chi|>x}}\frac{1}{|\rho_\chi|^2}
\ll
x(\log qx)^2,
\]
by the use of Lemma \ref{Lem:zero_sum2} in the last estimate.
Substituting this estimate into \eqref{EQ:pre_I_goal},
we obtain \eqref{EQ:I_goal} for $I(\chi)$.

\end{proof}

Now Theorem \ref{Asymptotic_B_SP} can be proven.
\begin{proof}[Proof of Theorem \ref{Asymptotic_B_SP}]
By the orthogonality of characters, we have
\[
S(x;q,a,b)
=
\frac{1}{\varphi(q)^2}\sum_{\chi_1,\chi_2\,(\md q)}
\overline{\chi_1(a)}\overline{\chi_2(b)}S(x;\chi_1,\chi_2).
\]
Thus, by Lemma \ref{pre_lem}, it suffices to show that
\[
\frac{1}{\varphi(q)^2}\sum_{\chi_1,\chi_2\,(\md q)}
|R(x;\chi_1,\chi_2)|
\ll
x^{2B_q^\ast}(\log qx)^5.
\]
By the Cauchy--Schwarz inequality and Lemma \ref{J_estimate},
the right-hand side above is
\[
\ll
\frac{1}{\varphi(q)^2}\sum_{\chi_1,\chi_2\,(\md q)}
J(\chi_1)^{1/2}J(\chi_2)^{1/2}
\ll
x^{2B_q^\ast}(\log qx)^5.
\]
This completes the proof.
\end{proof}

We move on to the proof of Theorem \ref{Asymptotic_1B}.
At first, it might seem that we can obtain this asymptotic formula
by summing up Theorem \ref{Asymptotic_B_SP} over residues.
However this procedure violates the uniformity over $q$ and so
instead  we take advantage of the ``bilinear nature''
of the error term $R(x;\chi_1,\chi_2)$ in Lemma \ref{pre_lem}.
With this in mind we prove the following 
which will be used in its full generality in the proof of Theorem \ref{thmsiegel}.

\begin{lem}\label{character_sum}
For positive integers $c,q$ and a character $\chi\ (\md q)$,
we have
\[
\sum_{\substack{a=1\\(a(c-a),q)=1}}^q\kern-10pt\chi(a)
=
\mu(q^\ast)\chi^\ast(c)\frac{\varphi(q)}{\varphi(q^\ast)}
\prod_{\substack{p\mid q\\p\nmid q^\ast c}}
\frac{p-2}{p-1},
\]
where $\chi^\ast\ (\md q^\ast)$ is the primitive character
which induces $\chi$.
\end{lem}
\begin{proof} By using the Chinese Remainder Theorem
and decomposing the character into the product of characters of prime power moduli
it is sufficient to prove the lemma in the case where $q$ is a prime power,
say $q=p^k$ and $q^\ast=p^\ell$.
If $\ell=0$, then
\[
\sum_{\substack{a=1\\(a(c-a),q)=1}}^q\kern-10pt\chi(a)
=
\sum_{\substack{a=1\\(a(c-a),p)=1}}^{p^k}1
=
\sum_{\substack{a=1\\a\not\equiv0,c\,(\md p)}}^{p^k}1
=
\left\{
\begin{array}{ll}
p^{k-1}(p-1)&(\text{if $p\mid c$}),\\
p^{k-1}(p-2)&(\text{if $p\nmid c$}),
\end{array}
\right.
\]
which coincides with the assertion.
If $\ell\ge1$, then we have
\[
\sum_{\substack{a=1\\(a(c-a),q)=1}}^q\kern-10pt\chi(a)
=
\sum_{\substack{a=1\\(a(c-a),p)=1}}^{p^k}\chi^\ast(a)
=
p^{k-\ell}
\sum_{\substack{a=1\\a\not\equiv0,c\,(\md p)}}^{p^\ell}\chi^\ast(a).
\]
If $\ell=1$, then by the orthogonality, this is
\[
=
p^{k-1}
\left(\sum_{a=1}^{p}\chi^\ast(a)-\sum_{\substack{a=1\\a\equiv c\,(\md p)}}^{p}\chi^\ast(a)\right)
=
\mu(p)\chi^\ast(c)\frac{\varphi(p^k)}{\varphi(p)},
\]
which also coincides with the assertion. If $\ell\ge2$, then by Theorem 9.4 of \cite{MV} we have
\[
\sum_{\substack{a=1\\(a(c-a),q)=1}}^q\kern-10pt\chi(a)
=
p^{k-\ell}
\left(\sum_{a=1}^{p^\ell}\chi^\ast(a)-\sum_{\substack{a=1\\a\equiv c\,(\md p)}}^{p^\ell}\chi^\ast(a)\right)
=0,
\]
which again satisfies the claimed equality. 
\end{proof}

\begin{proof}[Proof of Theorem \ref{Asymptotic_1B}]
By using the symmetry between $\ell$ and $m$, we have
\[
\sum_{\substack{n\le x\\n\equiv c\,(\md q)}}\hspace{-2mm}G(n)
=
\sum_{\substack{\ell+m\le x\\\ell+m\equiv c\,(\md q)\\(\ell m,q)=1}}
\Lambda(\ell)\Lambda(m)
+
O\Bigg(\sum_{\substack{\ell+m\le x\\\ell+m\equiv c\,(\md q)\\(m,q)>1}}
\Lambda(\ell)\Lambda(m)\Bigg).
\]
This error term can be estimated as
\begin{align*}
\sum_{\substack{\ell+m\le x\\\ell+m\equiv c\,(\md q)\\(m,q)>1}}
\Lambda(\ell)\Lambda(m)
\ll
\sum_{\ell\le x}\Lambda(\ell)
\sum_{\substack{m\le x\\(m,q)>1}}\Lambda(m)
\ll
x(\log qx)^2,
\end{align*}
using the same estimate as in \eqref{EQ:remove_coprimality}.
Thus it suffices to consider
\[
\sum_{\substack{\ell+m\le x\\\ell+m\equiv c\,(\md q)\\(\ell m,q)=1}}
\Lambda(\ell)\Lambda(m)
=
\frac{1}{\varphi(q)^2}
\sum_{\substack{a=1\\(a(c-a),q)=1}}^q
\sum_{\chi_1,\chi_2\,(\md q)}
\overline{\chi_1(a)}\overline{\chi_2(c-a)}S(x;\chi_1,\chi_2).
\]
We apply Lemma \ref{pre_lem} to the right-hand side, and evaluate the resulting
expression. Clearly,
\[
\frac{1}{\varphi(q)^2}
\sum_{\substack{a=1\\(a(c-a),q)=1}}^{q}
\sum_{\chi_1,\chi_2\,(\md q)}
\delta_0(\chi_2)(1+\delta_1(\chi_1)q^{1/2})x(\log qx)^2
\ll
x(\log qx)^2.
\]

Also,
\begin{align}
\label{sieveidentity}
\sum_{\substack{a=1\\(a(c-a),q)=1}}^q1
=
\varphi(q)^2\mathfrak{S}_{q}(c)
\end{align}
by Lemma \ref{character_sum} with the principal character.
Therefore, it suffices to show that
\[
R
=
\frac{1}{\varphi(q)^2}
\sum_{\substack{a=1\\(a(c-a),q)=1}}^q
\sum_{\chi_1,\chi_2\,(\md q)}
\overline{\chi_1(a)}\overline{\chi_2(c-a)}R(x;\chi_1,\chi_2)
\ll
x^{2B_q^\ast}(\log qx)^5.
\]
We have
\begin{gather*}
\sum_{\chi_1,\chi_2\,(\md q)}
\overline{\chi_1(a)}\overline{\chi_2(c-a)}R(x;\chi_1,\chi_2)\\
=
\int_0^1
\left(\sum_{\chi_1\,(\md q)}\overline{\chi_1(a)}W(\alpha,\chi_1)\right)
\left(\sum_{\chi_2\,(\md q)}\overline{\chi_2(c-a)}W(\alpha,\chi_2)\right)
T(-\alpha)d\alpha.
\end{gather*}
Next the Cauchy--Schwarz inequality gives
\[
R
\ll
\frac{1}{\varphi(q)^2}\sum_{\substack{a=1\\(a(c-a),q)=1}}^q
\int_{-1/2}^{1/2}
\left|\sum_{\chi\,(\md q)}\overline{\chi(a)}W(\alpha,\chi)\right|^2
|T(\alpha)|d\alpha.
\]
By the orthogonality of characters, we have
\begin{align*}
\sum_{\substack{a=1\\(a(c-a),q)=1}}^q
\left|\sum_{\chi\,(\md q)}\overline{\chi(a)}W(\alpha,\chi)\right|^2
&\ll
\sum_{\substack{a=1\\(a,q)=1}}^q
\left|\sum_{\chi\,(\md q)}\overline{\chi(a)}W(\alpha,\chi)\right|^2\\
&=
\varphi(q)\sum_{\chi\,(\md q)}|W(\alpha,\chi)|^2.
\end{align*}
Thus, by Lemma \ref{J_estimate},
\[
R(\chi)
\ll
\frac{1}{\varphi(q)}\sum_{\chi\,(\md q)}J(\chi)
\ll
x^{2B_q^\ast}(\log qx)^5.
\]
Summing up the above calculations,
we complete the proof.
\end{proof}

\section{The connection between $S(x;q,a,b)$ and GRH}\label{Dirichlet_series}
Consider the Dirichlet series
\[
F(s)=F(s;q,a,b)=\sum_{n=1}^{\infty}\frac{G(n;q,a,b)}{n^s},
\]
which converges absolutely and is analytic for $\sigma>2$.
Analytic properties of $F(s)$ have been studied by Egami and the third author \cite{EM},
the first author and Schlage-Puchta \cite{B_SP11}
(in the case $q=1$), and by R{\"u}ppel \cite{Rup12} (general case).
In particular, the connection between $S(x;q,a,b)$ and GRH
can be understood through the analytic continuation of $F(s)$.
We first find the meromorphic continuation of $F(s)$
via Theorem \ref{Asymptotic_B_SP} in the following proposition.    This type of result , 
under GRH was obtained in \cite{EM}, \cite{Rup12}.
\begin{prop}\label{Ruppel_continuation}
The function $F(s)$ can be continued meromorphically
to the half plane $\sigma>2B_q$.
Its 
poles in the half plane $\sigma>2B_q$ are
\begin{enumerate}[(i)] 
\item a simple pole at $s=2$ with residue $\varphi(q)^{-2}$,
\item a possible pole
at $s=\rho_q+1$ of at most order 1 with residue
\[
r(\rho_q)
=
-\frac{1}{\varphi(q)^2}
\frac{1}{\rho_q}
\sum_{\substack{\chi\,(\md q)\\L(\rho_q,\chi)=0}}
(\overline{\chi(a)}+\overline{\chi(b)})m_\chi(\rho_q),
\]
where $\rho_q$ is a zero of 
$\prod_{\chi\,(\md q)}L(s,\chi)$
with $0<\Re \rho_q<1$ and $m_\chi(\rho_q)$ is the multiplicity of $\rho_q$
as a zero of $L(s,\chi)$.
\end{enumerate}
In particular, assuming {\upshape DZC}, $B_q<1$
and $\overline{\chi}(a)+\overline{\chi}(b)\neq0$ for all $\chi\ (\md q)$ we obtain

\[
1+B_q=\inf\Set{\sigma_0\ge\frac{3}{2}|
F(s)-\frac{1}{\varphi(q)^2}\frac{1}{s-2}\text{ is analytic on $\sigma>\sigma_0$}}.
\]
\end{prop}

\begin{proof}
From Theorem \ref{Asymptotic_B_SP}, we have
\begin{equation}\label{asymp_continuation}
S(x;q,a,b)
=\frac{x^2}{2\varphi(q)^2}
+\sum_{\rho_q}r(\rho_q)\frac{x^{\rho_q+1}}{\rho_q+1}
+E(x;q,a,b)
\end{equation}
for $x\geq 1$, where
\begin{equation}\label{error_continuation}
E(x;q,a,b)\ll x^{2B_q}(\log 2qx)^5
\end{equation}
since $B_q^\ast\le B_q$.
For $\sigma>2$, we have
\begin{align*}
F(s)
=
\int_1^\infty u^{-s}dS(u;q,a,b)
=
s\int_1^{\infty}S(u;q,a,b)u^{-s-1}du
\end{align*}
with integration by parts, because $S(x;q,a,b)=0$ for $x<4$.
Substitute \eqref{asymp_continuation} in the right-hand side of the above.
The swapping of summation and integration is justified
due to absolute convergence.    Therefore we have
\begin{align}
\label{rhs_of_F}
F(s)
&=\frac{s}{2\varphi(q)^2(s-2)}+\sum_{\rho_q}\frac{r(\rho_q)s}{(\rho_q+1)(s-\rho_q-1)}
+s\int_1^\infty E(u;q,a,b)u^{-s-1}du\\
&=\frac{1}{\varphi(q)^2(s-2)}+\sum_{\rho_q}\frac{r(\rho_q)}{s-\rho_q-1}
+s\int_1^\infty E(u;q,a,b)u^{-s-1}du+C_1(q,a,b),\notag
\end{align}
where
\[
C_{1}(q,a,b)=\frac{1}{2\varphi(q)^{2}}+\sum_{\rho_{q}}  \frac{r(\rho_{q})}{\rho_{q}+1}
\]
and the sum  converges  due to Lemma \ref{Lem:zero_sum2} to yield a certain constant depending on $a,b,q$.
The sum on the right-hand side of \eqref{rhs_of_F} converges 
uniformly
for $s\in\mathbb{C}\setminus\{\rho_q+1\}$ 
and determines a meromorphic function
on $\mathbb{C}$.
(Since for $|\Im \rho_{q}|=|\gamma_{\chi}|>2|\Im s|$, we have $|\rho_{q}+1-s|\geq
|\gamma_{\chi}-\Im s|\geq |\gamma_{\chi}|/2$, then the compact
uniformly convergence can be justified by Lemma \ref{Lem:zero_sum2}.
Further note that in Lemma \ref{Lem:zero_sum2}, each zero appears with multiplicity.)
The first and second term on the right-hand side of \eqref{rhs_of_F}
already give the announced residues of the proposition.
Using the estimate \eqref{error_continuation} we see that the integral
\[
\int_1^\infty E(u;q,a,b)u^{-s-1}du
\]
converges 
uniformly on the half plane $\sigma>2B_q$ and so it
defines an analytic function on $\sigma>2B_q$.
This completes the proof of the 
meromorphic
continuation.

For the last assertion,
the inequality
\[
1+B_q\ge\inf\Set{\sigma_0\ge\frac{3}{2}|
F(s)-\frac{1}{\varphi(q)^2}\frac{1}{s-2}\text{ is analytic on $\sigma>\sigma_0$}}
\]
follows  from the above 
meromorphic
 continuation, since $1+B_{q}\geq 2B_{q}$.
We next prove the reverse inequality.
If $B_q=1/2$, then the implication is trivial.
Hence we can assume that 
 $1/2<B_q<1$ and we have $\max(2B_q,3/2)<1+B_q$,
so that we can take $\varepsilon>0$ such that
$\max(2B_q,3/2)<1+B_q-\varepsilon$.
By the definition of $B_q$, we can find a zero $\rho_q$ such that
$1/2<B_q-\varepsilon<\Re \rho_q$.
Then by the above 
meromorphic continuation,
we have a possible pole of $F(s)$ of residue $r(\rho_q)$ at $\rho_q+1$.
(Note that we do not necessarily have meromorphic continuation
on $\sigma>1+B_{q}-\varepsilon$ if $B_{q}=1$, since then
$B_{q}+1-\varepsilon=2B_{q}-\varepsilon<2B_{q}$.)
By DZC and
the assumption that $\Re\rho_q>1/2$ we have
\[
r(\rho_q)=-\frac{1}{\varphi(q)^2}\frac{1}{\rho_q}
(\overline{\chi(a)}+\overline{\chi(b)})m,
\]
where $m\ge1$.
Since we have assumed
that $\overline{\chi}(a)+\overline{\chi}(b)\neq0$ for all $\chi\ (\md q)$,
this residue is non-zero so that $\rho_q+1$ is a pole of $F(s)$
in the half plane $\sigma>1+B_q-\varepsilon>3/2$.
This implies
\[
1+B_q-\varepsilon\le\inf\Set{\sigma_0\ge\frac{3}{2}|
F(s)-\frac{1}{\varphi(q)^2}\frac{1}{s-2}\text{ is analytic on $\sigma>\sigma_0$}}
\]
so that on letting $\varepsilon\to0$ we obtain the reverse inequality.
\end{proof}

We can now prove Theorem \ref{MainTheorem-1}.
\begin{proof}[Proof of Theorem \ref{MainTheorem-1}]
First we prove the assertion (1).
If $q>x^{(1-B_q)/2}(\log x)$, then the left-hand side of \eqref{MainTheorem-1-asymp} is
\begin{align*}
\le
\left(\sum_{\substack{\ell\le x\\\ell\equiv a\,(\md q)}}\Lambda(\ell)\right)
\left(\sum_{\substack{m\le x\\m\equiv b\,(\md q)}}\Lambda(m)\right)
\ll
\frac{(x\log x)^2}{q^2}+(\log x)^2
\ll
x^{1+B_q}.
\end{align*}
Also the first term on the right-hand side of \eqref{MainTheorem-1-asymp} is $\ll x^{1+B_q}$.
Thus \eqref{MainTheorem-1-asymp} holds trivially.
Therefore, we may assume $q\le x^{(1-B_q)/2}(\log x)$.
By using Lemma \ref{Lem:zero_sum2},
\begin{align*}
&\frac{1}{\varphi(q)^2}\sum_{\chi\,(\md q)}(\overline{\chi(a)}+\overline{\chi(b)})
\sum_{\rho_{\chi}}
\frac{x^{\rho_{\chi}+1}}{\rho_{\chi}(\rho_{\chi}+1)}\\
{}\ll{}&
\frac{x^{B_q+1}}{\varphi(q)^2}\sum_{\chi\,(\md q)}
\left((\log 2q)^2+\delta_1(\chi)q^{1/2}(\log q)^2\right)
{}\ll{}
x^{1+B_q}.
\end{align*}
Therefore, by Theorem \ref{Asymptotic_B_SP}, we obtain
\[
S(x;q,a,b)=\frac{x^2}{2\varphi(q)^2}+O(x^{1+B_q}+x^{2B_q^\ast}(\log x)^5).
\]
If $1-B_q>5\log\log x/\log x$, then we see that
\[
x^{2B_q^\ast}(\log x)^5\le x^{1+B_q}\cdot x^{-(1-B_q)}(\log x)^5\le x^{1+B_q}
\]
so that \eqref{MainTheorem-1-asymp} follows. Thus we may assume $1-B_q\le5\log\log x/\log x$.
Further using the assumption on $q$ we find that
$q\le x^{(1-B_q)/2}(\log x)\le (\log x)^{7/2}$.
Thus, by recalling the definition of $\eta=\eta_{q}(x)$ and choosing $\varepsilon=1/7$, we obtain
\[
B_q^\ast
\le 1-\eta\le 1-\frac{c_1}{\max((\log x)^{1/2},(\log x)^{4/5})}
=1-\frac{c_1}{(\log x)^{4/5}},
\]
where $c_1=c_1(1/7)$ is an absolute constant. This gives
\[
x^{2B_q^\ast}(\log x)^5\le x^{1+B_q^\ast}(\log x)^5\exp(-c_1(\log x)^{1/5})\ll x^{1+B_q}.
\]
Therefore we always arrive at \eqref{MainTheorem-1-asymp} .

We next prove the assertion (2).
Assume that the formula \eqref{MainTheorem-1-formula} holds, i.e.,
\begin{equation}\label{assumed_error}
S(x;q,a,b)=\frac{x^2}{2\varphi(q)^2}+E_d(x),\quad
E_d(x)\ll_qx^{1+d+\varepsilon}
\end{equation}
for arbitrary $\varepsilon>0$.
Now we use this formula to obtain the meromorphic continuation of $F(s)$.
In the same manner as  in the proof of Proposition \ref{Ruppel_continuation}
we have
\[
F(s)-\frac{1}{\varphi(q)^2}\frac{1}{s-2}
=
s\int_1^{\infty}E_d(u)u^{-s-1}du+\frac{1}{2\varphi(q)^2}
\]
for $\sigma>2$.
Then, by \eqref{assumed_error},
the right-hand side gives an analytic function on $\sigma>1+d$.
Therefore under the last assertion of Proposition \ref{Ruppel_continuation},
we have $B_q\le d$ provided DZC, $B_q<1$
and that $\overline{\chi(a)}+\overline{\chi(b)}\neq0$ for any $\chi\,(\md q)$.
The supplement for $a=b$ is proved in Section~\ref{sec:exclusion}.
This now completes the proof. 
\end{proof}

We next move on to Theorem \ref{thmsiegel}.
The  strategy is the same as in the proof of Theorem \ref{MainTheorem-1}.
We consider the Dirichlet series
\[
F_1(s)
=
F_1(s;q;c)
=
\sum_{\substack{n=1\\n\equiv c\,(\md q)}}^\infty\frac{G(n)}{n^s}.
\]
The meromorphic continuation of $F_1(s)$ is obtained
via Theorem \ref{Asymptotic_1B}.

\begin{prop}\label{Ruppel_continuation_1B}
The function $F_1(s)$ can be continued meromorphically
to the half plane $\sigma>2B_q$.
Its  poles in the half plane $\sigma>2B_q$ are
\begin{enumerate}[(i)] 
\item a possible pole at $s=2$ of order at most 1 with residue $\mathfrak{S}_q(c)$,
\item a possible pole at $s=\rho_q+1$ of at most order 1 with residue
\[
r_1(\rho_q)
=
-\frac{2}{\varphi(q)^2}\frac{1}{\rho_q}
\sum_{\substack{\chi\,(\md q)\\L(\rho_q,\chi)=0}}m_\chi(\rho_q)
\sum_{\substack{a=1\\(a(c-a),q)=1}}^q\overline{\chi(a)},
\]
where $\rho_q$ is a zero of $\prod_{\chi\,(\md q)}L(s,\chi)$
with $0<\Re \rho<1$.
\end{enumerate}
\end{prop}

\begin{proof}
This can be proven in the same manner
as Proposition \ref{Ruppel_continuation}.

\end{proof}

We also require the following lemma.
\begin{lem}
\label{lem:singular_series_average}
For $x\ge2$ and positive integers $c,q$, we have
\[
\sum_{\substack{n\le x\\n\equiv c\,(\md q)}}J(n)=\frac{\mathfrak{S}_q(c)}{2}x^2+O(x\log x),
\]
where the implicit constant is absolute.
\end{lem}
\begin{proof}
We first consider the case of $(2,q)\nmid c$, i.e. $q$ is even while $c$ is odd.
Since there is no even number $n\equiv c\ (\md q)$
the sum on the left-hand side is $=0$ since $J(n)=0$ for odd $n$.
Also, $\mathfrak{S}_{q}(c)=0$ by definition in this case.
Thus the assertion trivially holds for $(2,q)\nmid c$.
We next consider the case $(2,q)\mid c$.
We use an expression
\[
J(2N)=2C_2N\sum_{\substack{d|N\\d:\,\text{odd}}}\frac{\mu(d)^2}{\varphi_2(d)},\quad
\varphi_2(n)=\prod_{p\mid n}(p-2),
\]
to obtain 
\begin{equation}
\label{singular_series_average:swapped}
\sum_{\substack{n\le x\\n\equiv c\,(\md q)}}J(n)
=
\sum_{\substack{2N\le x\\2N\equiv c\,(\md q)}}J(2N)
=
2C_2\sum_{\substack{d\le x\\d:\,\text{odd}}}\frac{\mu(d)^2d}{\varphi_2(d)}
\sum_{\substack{2dn\le x\\2dn\equiv c\,(\md q)}}n.
\end{equation}
Let $q_1=q/(2d,q)$. The congruence
\[
2dn\equiv c\ (\md q)
\]
has a solution, say $c_1\ (\md q_1)$ if $(2d,q)\mid c$, and no solution if $(2d,q)\nmid c$.
Moreover, the condition $(2d,q)\mid c$ is equivalent to $(d,q)\mid c$
since $d$ is odd and $(2,q)\mid c$.
Hence from \eqref{singular_series_average:swapped}, we have
\begin{equation}
\label{singular_series_average:inner_sum_calculated}
\begin{aligned}
\sum_{\substack{n\le x\\n\equiv c\,(\md q)}}J(n)
&=
2C_2\sum_{\substack{d\le x/2\\d:\,\text{odd}\\[.3mm](d,q)\mid c}}\frac{\mu(d)^2d}{\varphi_2(d)}
\sum_{\substack{n\le x/2d\\n\equiv c_1\,(\md q_1)}}n\\
&=
\frac{C_2x^2}{4q}\sum_{\substack{d\le x/2\\d:\,\text{odd}\\[.3mm](d,q)\mid c}}\frac{\mu(d)^2(2d,q)}{d\varphi_2(d)}
+
O\left(x\sum_{\substack{d\le x/2\\d:\,\text{odd}}}\frac{\mu(d)^2}{\varphi_2(d)}\right).
\end{aligned}
\end{equation}
As for the second term on the right-hand side of \eqref{singular_series_average:inner_sum_calculated},
we have
\begin{equation}
\label{singular_series_average:phi2_average}
\sum_{\substack{d\le x/2\\d:\,\text{odd}}}\frac{\mu(d)^2}{\varphi_2(d)}
\le
\prod_{2<p\le x}\left(1+\frac{1}{p-2}\right)
\ll\log x.
\end{equation}
 For the first term on the right-hand side of \eqref{singular_series_average:inner_sum_calculated},
we have
\begin{equation}
\label{singular_series_average:main_term}
\frac{C_2}{4q}
\sum_{\substack{d\le x/2\\d:\,\text{odd}\\[.3mm](d,q)\mid c}}\frac{\mu(d)^2(2d,q)}{d\varphi_2(d)}
=
\frac{C_2\cdot(2,q)}{4q}
\sum_{\substack{d:\,\text{odd}\\[.3mm](d,q)\mid c}}\frac{\mu(d)^2(d,q)}{d\varphi_2(d)}
+
O\left(
\sum_{\substack{d>x/2\\d:\,\text{odd}}}\frac{\mu(d)^2}{d\varphi_2(d)}
\right).
\end{equation}
This remainder term is estimated by using \eqref{singular_series_average:phi2_average} as
\begin{equation}
\label{singular_series_average:main_tail}
\sum_{\substack{d>x/2\\d:\,\text{odd}}}\frac{\mu(d)^2}{d\varphi_2(d)}
=
\sum_{\substack{d>x/2\\d:\,\text{odd}}}\frac{\mu(d)^2}{\varphi_2(d)}\int_{d}^{\infty}\frac{du}{u^2}
\le
\int_{x/2}^\infty\left(\sum_{\substack{d\le u\\d:\,\text{odd}}}\frac{\mu(d)^2}{\varphi_2(d)}\right)\frac{du}{u^2}
\ll
\frac{\log x}{x}.
\end{equation}
On the other hand we have
\begin{align*}
\frac{C_2\cdot(2,q)}{4q}
\sum_{\substack{d:\,\text{odd}\\[.3mm](d,q)\mid c}}\frac{\mu(d)^2(d,q)}{d\varphi_2(d)}
=
\frac{C_2\cdot(2,q)}{4q}
\prod_{\substack{p>2\\p\mid (q,c)}}\frac{p-1}{p-2}
\prod_{\substack{p>2\\p\nmid q}}\frac{(p-1)^2}{p(p-2)}.
\end{align*}
Since the definition of $C_2$ is
$$
C_2=2\prod_{p>2}\left(1-\frac{1}{(p-1)^2}\right)=2\prod_{p>2}\frac{p(p-2)}{(p-1)^2},
$$
the right-hand side of the above is equal to
\begin{align}
\label{singular_series_average:local_singular_series}
=&\ 
\frac{(2,q)}{2q}
\prod_{\substack{p>2\\p\mid (q,c)}}\frac{p}{p-1}
\prod_{\substack{p>2\\p\mid q\\p\nmid c}}\frac{p(p-2)}{(p-1)^2}
=\frac{\mathfrak{S}_{q}(c)}{2}
\end{align}
since $(2,q)\mid c$.
Substituting \eqref{singular_series_average:main_tail}
and \eqref{singular_series_average:local_singular_series}
into \eqref{singular_series_average:main_term},
we have
\[
\frac{C_2}{4q}
\sum_{\substack{d\le x\\d:\,\text{odd}\\[.3mm](d,q)\mid c}}\frac{\mu(d)^2(2d,q)}{d\varphi_2(d)}
=
\frac{\mathfrak{S}_{q}(c)}{2}+O\left(\frac{\log x}{x}\right).
\]
Combining this with \eqref{singular_series_average:phi2_average}
and \eqref{singular_series_average:inner_sum_calculated}, we obtain the lemma.
\end{proof}

We finally prove Theorem \ref{thmsiegel}.
\begin{proof}[Proof of Theorem \ref{thmsiegel}]
We first prove (1).
By using Lemma \ref{character_sum},
the second term on the right-hand side of Theorem \ref{Asymptotic_1B} is
\begin{align*}
&=
\frac{2}{\varphi(q)^2}
\sum_{\chi\,(\md q)}
\sum_{\rho_\chi}\frac{x^{\rho_\chi+1}}{\rho_\chi(\rho_\chi+1)}
\left(\sum_{\substack{a=1\\(a(c-a),q)=1}}^q\overline{\chi(a)}\right)\\
&\ll
\frac{x^{B_q+1}}{\varphi(q)}
\sum_{\chi\,(\md q)}\frac{1}{\varphi(q^\ast)}
\sum_{\rho_\chi}\frac{1}{|\rho_\chi(\rho_\chi+1)|}.
\end{align*}
By Lemma \ref{Lem:zero_sum2}, this can be estimated as
\begin{align*}
\ll
x^{1+B_q}
\frac{q^{2/3}}{\varphi(q)}
\sum_{\chi\,(\md q)}\frac{1}{\varphi(q^\ast)}
&\ll
x^{1+B_q}
\frac{q^{2/3}}{\varphi(q)}
\sum_{q^{\ast}|q}\frac{1}{\varphi(q^{\ast})}\sideset{}{^\ast}\sum_{\chi\,(\md q^{\ast})}1\\
&\le
x^{1+B_q}
\frac{q^{2/3}\tau(q)}{\varphi(q)}
\ll
x^{1+B_q}
\end{align*}
where the summation symbol with $\ast$ denotes the sum over primitive characters
and $\tau(q)$ denotes the number of divisors of $q$.
Thus Theorem \ref{Asymptotic_1B} gives
\[
\sum_{\substack{n\le x\\n\equiv c\,(\md q)}}G(n)
=
\frac{\mathfrak{S}_q(c)}{2}x^2
+O(x^{1+B_q}+x^{2B_q^\ast}(\log qx)^5).
\]
Then an argument similar to that for (1) of Theorem \ref{MainTheorem-1}
gives \eqref{thmsiegel_asymp}.

We next prove (2).
Assume that formula \eqref{eq:thmsiegel} holds.
Then by Lemma \ref{lem:singular_series_average},
\begin{equation}\label{thmsiegel:assumed_error}
\sum_{\substack{n\le x\\ n\equiv c\,(\md q)}}G(n)
=\frac{\mathfrak{S}_q(c)}{2}x^2+E_d(x),\quad
E_d(x)\ll x^{1+d+\varepsilon}
\end{equation}
for arbitrary $\varepsilon>0$.
As in the case of $F(s)$ we can obtain the meromorphic continuation of $F_1(s)$
to the half plane $\sigma>1+d$, which has only one possible pole at $s=2$.
We compare this analytic continuation with Proposition \ref{Ruppel_continuation_1B}.
By assumption (a) we have $B_q=\Re\rho_0<1$ so that $2B_q<1+\Re\rho_0$.
Thus by using assumption (b) $F_1(s)$ has a possible pole of order $\le1$ with residue
\[
-\frac{2}{\varphi(q)^2}\frac{m}{\rho_0}\sum_{\substack{a=1\\(a(c-a),q)=1}}^{q}\overline{\chi(a)},
\]
where $m\ge1$. By Lemma \ref{character_sum},
we find that this residue is non-zero provided under the assumption (c):\ 
$q^\ast$ is squarefree, $(c,q^\ast)=1$, 
and yet another assumption
\[2\nmid q\quad\text{or}\quad2|q^{\ast}c,\]
the last of which is assured by the condition $(2,q)\mid c$ of Theorem \ref{thmsiegel}.
Therefore $\rho_0+1$ is a pole of $F_1(s)$.
By comparing the position of this pole and the analytic continuation
we have $1+B_q\le1+d$. This completes the proof.
\end{proof}

\section{Exclusion of  $B_{q}=1$ for $a=b$}\label{sec:exclusion}

In this last section, we exclude the possibility of $B_{q}=1$ for
$a=b$ in Theorem~\ref{MainTheorem-1} (2)
following an idea of Ruzsa.

\medskip
 Let $G_{a,q}(n)=G(n;q,a,a)$, so that
 $S(x;q,a,a)= \sum_{n\leq x} G_{a,q}(n)$.

Then our assumption \eqref{MainTheorem-1-formula} 
in Theorem~\ref{MainTheorem-1} (2) reads
\begin{equation}
  \label{eq:mainassump}
 S(x;q,a,a)=\frac{x^{2}}{2\ph(q)^{2}}+O_{q}(x^{1+d+\veps})  
\end{equation}
for some $1/2\leq d<1$ and any $\veps>0$.
We prove that \eqref{eq:mainassump} together with DZC
implies that $B_{q}<1$.

\begin{proof}
  
\textit{Step 1.} For $|z|<1$ let 
\[
F_{a,q}(z)=\sum_{\substack{n\geq 1\\n\equiv a \;(q)}}
\Lambda (n) z^{n},
\text{ so } F_{a,q}^{2}(z)=\sum_{n\geq 1} G_{a,q}(n)z^{n}.
\]
Then, since $(1-z)^{-1}=1+z+z^2+\cdots$, we obtain an identity
\[
\frac{1}{1-z}F_{a,q}^{2}(z)
=
\sum_{n\ge1}\left(\sum_{\substack{i+j=n\\i\ge1,\,j\ge0}}G_{a,q}(i)\right)z^n
=
\sum_{n\ge1}S(n;q,a,a)z^n.
\]
From \eqref{eq:mainassump} we deduce that
\begin{equation}
\label{eq:tGsum}
\frac{1}{1-z}F_{a,q}^{2}(z)
=
\frac{1}{2\ph(q)^2}\sum_{n\geq 1}n^2z^{n}
+
O_{q}\left(\sum_{n\geq1}n^{1+d+\veps}|z|^{n}\right).
\end{equation}
Using the derivative of the geometric series twice we find that
\[
\frac{2}{(1-z)^{3}}=\sum_{n\geq 1}n^{2}z^{n}+O\left(\sum_{n=1}^{\infty}n|z|^n\right)
\]
for $|z|<1$ so we can evaluate the main term in \eqref{eq:tGsum}.
The above error terms are estimated with the help of the following Lemma.
\begin{lem}
\label{lem:C}
For a sequence of positive real numbers $(a_n)_{n=0}^{\infty}$ satisfying
\begin{equation}
\label{EQ:A_cond_Abel}
A(x):=\sum_{n\le x}a_n\le Cx^\kappa
\end{equation}
for all $x\ge0$ with some constants $C,\kappa\ge0$, we have
\begin{equation}
\label{eq:dpleps}
\sum_{n\ge0}a_ne^{-n/X}\le C\Gamma(\kappa+1)X^\kappa.
\end{equation}
for any real number $X\ge1$.
\end{lem}
\begin{proof}
By partial summation 
and \eqref{EQ:A_cond_Abel}, 
the above series is expressed as
\[
\sum_{n\ge0}a_ne^{-n/X}
=
\frac{1}{X}\int_{0}^{\infty}e^{-u/X}A(u)du.
\]
Also by
\eqref{EQ:A_cond_Abel}, we estimate this integral by
\[
\le
\frac{C}{X}\int_{0}^{\infty}e^{-u/X}u^{\kappa}du
=
CX^{\kappa}\int_{0}^{\infty}e^{-u}u^{\kappa}du
=
C\Gamma(\kappa+1)X^{\kappa}.
\]
Thus the lemma follows.
\end{proof} 

In what follows we work on the circle $|z|=R$ with $R=e^{-1/N}$ for a large positive integer $N$.
Since $\sum_{n\le x}n^{1+d+\epsilon}\ll x^{2+d+\epsilon}$,
by Lemma \ref{lem:C}, we continue \eqref{eq:tGsum} by
\[
\frac{1}{1-z}F_{a,q}^{2}(z)
=
\frac{1}{\ph(q)^2(1-z)^{3}}
+
O_{q}\left(N^C\right),\quad
C:=2+d+\veps<3
\]
on the circle $|z|=R$. Therefore, we obtain
\[
F_{a,q}(z)^{2}=\frac{1}{(1-z)^{2}\ph(q)^{2}}+O_{q}\left(|1-z|N^{C}\right).
\]
Note that the second term on the right-hand side is smaller than the first term if
\begin{equation}
\label{eq:major}
|1-z|\leq cN^{-C/3}
\end{equation}
with sufficiently small constant $c>0$ depending only on $q$ and $d$.
Thus on the arc $|z|=R$ with \eqref{eq:major}, which we may call a major arc on $|z|=R$,
we can take the complex square root of the formula for
$F_{a,q}(z)$ which yields
\begin{equation}
\label{eq:squareroot}
F_{a,q}(z)
=
\plm \frac{1}{(1-z)\ph(q)}+O_{q}(|1-z|^{2}N^{C})
\end{equation}
as an asymptotic formula for all $z$ on the major arc.
Here, the same sign $\pm$ is kept on the whole major arc
since $F_{a,q}(z)$ is continuous.
Because there are only non-negative coefficients,
the left-hand side in \eqref{eq:squareroot} 
is non-negative 
for the choice $z=e^{-1/N}$.
With this choice, the main term in \eqref{eq:squareroot} is
real and therefore must also be non-negative.
Therefore, the sign $\pm$ on the right-hand side of \eqref{eq:squareroot}
is $+$.

At this point we notice that we are  unable to obtain a similar asymptotic formula
when $a\neq b$. The square root step here
shows that using this method we can prove the exclusion of $B_{q}=1$ in
Theorem~\ref{MainTheorem-1} (2) only when $a=b$.

\textit{Step 2.}
Now we use the kernel
\begin{equation}
\label{EQ:K_kernel}
K(z)=z^{-N-1}+z^{-N}+\dots+z^{-2}=z^{-N-1}\frac{1-z^{N}}{1-z}.
\end{equation}
Then by using Cauchy's integral formula, we obtain
\[
\psi(N;q,a)
=
\frac{1}{2\pi i} \int_{|z|=R} F_{a,q}(z) K(z)dz,\quad
N
=
\int_{|z|=R}\frac{1}{1-z}K(z)dz.
\]
Thus we deduce that
\begin{equation}
\label{eq:psiaq}
\psi(N;q,a)
=
\frac{N}{\ph(q)}+\frac{1}{2\pi i}\int_{|z|=R}\left(F_{a,q}(z)-\frac{1}{(1-z)\ph(q)}\right)K(z)dz.
\end{equation}
From the second expression of $K(z)$ in \eqref{EQ:K_kernel},
we see that $K(z)\ll |1-z|^{-1}$.
Therefore, on the major arc \eqref{eq:major} of length $O(N^{-C/3})$ 
we insert the asymptotic formula
\eqref{eq:squareroot} and
the contribution to this integral is $O(N^{C-2C/3})=O(N^{C/3})$ with $C/3<1$.

\textit{Step 3.} For the rest of the circle the minor arc
 where $|1-z|>cN^{-C/3}$ 
we proceed with the Cauchy--Schwarz inequality
and apply the Parseval identity.

By the Parseval identity, we obtain the estimate over the full arc
\begin{equation}
\label{eq:fullc}
\begin{aligned}
\int_{|z|=R}\left|F_{a,q}(z)-\frac{1}{(1-z)\ph(q)}\right|^{2}dz
&\ll
\sum_{n\geq 0}(\Lambda(n)+1)^2e^{-2n/N}
\ll
N^{1+\veps}, 
\end{aligned}
\end{equation}
where we used the estimate
\[
\sum_{n\le x}(\Lambda(n)+1)^2\ll x\log x\ll x^{1+\veps}
\]
and Lemma~\ref{lem:C}.

On the other hand,
we use the decay of the kernel $K(z)$ on the minor arc.
By using the estimate $K(z)\ll|1-z|^{-1}$,
we start with
\[
I
=
\int_{\substack{|z|=R \\ |1-z|>cN^{-C/3}}} |K(z)|^{2}dz
\ll
\int_{\substack{|z|=R \\ |1-z|>cN^{-C/3}}}  \frac{dz}{|1-z|^{2}}.
\]
Now we use the parametrization $z=Re^{i\alpha}=e^{-1/N+i\alpha}$  with $-\pi\leq\alpha\leq\pi$.
On the minor arc, we have
\[
N^{-C/3} \ll |1-e^{-1/N+i\alpha}|\ll\left|-\frac{1}{N}+i\alpha\right|\ll\frac{1}{N}+|\alpha|
\]
so, by recalling the fact that $C<3$, 
we have
$|\alpha|\ge c_1N^{-C/3}$ with some small $c_1>0$ depending only on $q$ and $d$.
Also, note that
\begin{equation}
\label{EQ:explicit_modulus}
|1-z|^2=|1-e^{-1/N}\cos\alpha+ie^{-1/N}\sin\alpha|^2=1+e^{-2/N}-2e^{-1/N}\cos\alpha.
\end{equation}
By using the inequality of the arithmetic and geometric mean
\[
2\left|e^{-1/N}\cos\alpha\right|\le e^{-2/N}+(\cos\alpha)^2,
\]
we find that
\[
|1-z|^2\ge1-(\cos\alpha)^2=(\sin\alpha)^2.
\]
If $|\alpha|\ge\pi/2$, then $\cos\alpha<0$ so \eqref{EQ:explicit_modulus} implies $|1-z|^2\ge1$.
This yields
\begin{align*}
I
\ll 
\int_{c_1N^{-C/3}<|\alpha|\le\pi}\frac{d\alpha}{|1-Re^{i\alpha}|^{2}}
&\ll
\int_{c_1N^{-C/3}<|\alpha|\le\pi/2}\frac{d\alpha}{(\sin\alpha)^2}
+
2\int_{\pi/2<|\alpha|\le\pi}d\alpha\\
&\ll
\int_{c_1N^{-C/3}<|\alpha|\le\pi/2}\frac{d\alpha}{\alpha^2}+1
\ll
N^{C/3}.
\end{align*}

Putting everything together the Cauchy--Schwarz inequality gives
the minor arc estimate for the integral in \eqref{eq:psiaq} as
\[
\ll
\left( \int_{|z|=R}\left|F_{a,q}(z)-\frac{1}{(1-z)\ph(q)}\right|^{2}dz\right)^{1/2} I^{1/2} 
\ll
N^{1/2+C/6+\veps}
\]
with $1/2+C/6<1$. This together with the major arc estimate allows us to conclude that
\[
\psi(N;q,a)-\frac{N}{\ph(q)} \ll N^{\veps}(N^{C/3} + N^{1/2+C/6}).
\]
The exponent of $N$ is $<1$ for small $\veps>0$. In the explicit formula for
\[
\psi_{a,q}(N)=\frac{1}{\ph(q)}\sum_{\chi(q)}\bar{\chi}(a)\psi(N,\chi)
\]
that we obtain by inserting the explicit formula 
for $\psi(N,\chi)$ from
Lemma~\ref{Lem:explicit}
with $T=N$
and assuming DZC, no two terms $N^{\rho_{\chi}}/\rho_{\chi}$
will cancel out for different characters.
We conclude that $B_{q}<1$.
\end{proof}


%
%

\end{document}